\documentclass[a4paper, leqno]{amsart}
\usepackage{amssymb, amsmath, amsthm, amsrefs}
\usepackage{hyperref}
\usepackage{mathrsfs}
\usepackage{array}
\usepackage[matha, mathx]{mathabx}
\usepackage[shortlabels]{enumitem}
\usepackage{inputenc}
\usepackage{graphicx}
\usepackage{ dsfont }
\usepackage{xcolor}
\numberwithin{equation}{section}
\usepackage{amssymb, amsmath, amsthm, amsrefs}
\usepackage{hyperref}
\usepackage{mathrsfs}
\usepackage[matha, mathx]{mathabx}
\usepackage[shortlabels]{enumitem}

\usepackage{pict2e,picture}

\makeatletter
\DeclareRobustCommand{\bbDelta}{{\mathpalette\bb@Delta\relax}}
\newcommand{\bb@Delta}[2]{%
  \begingroup
  \sbox\z@{$\m@th#1\Delta$}%
  \dimendef\Dht=6 \dimendef\Dwd=8
  \setlength{\Dwd}{\wd\z@}%
  \setlength{\Dht}{\ht\z@}%
  \begin{picture}(\Dwd,\Dht)
  \put(0,0){$\m@th#1\Delta$}
  \put(.42\Dwd,.7\Dht){\line(10,-26){.25\Dht}}
  \end{picture}%
  \endgroup
}

\usepackage{xcolor}
\numberwithin{equation}{section}
\newtheorem{thm}{Theorem}

\newtheorem{lemma}[thm]{Lemma}
\newtheorem{proposition}[thm]{Proposition}
\theoremstyle{definition}

\newtheorem{rmk}[thm]{Remark}
\newtheorem{define}[thm]{Definition}

\newcommand{\abs}[1]{\left|#1\right|}

\title[Geometric description of some Loewner chains]{Geometric description of some Loewner chains with infinitely many slits}
\author[Eleftherios Theodosiadis]{E. K. Theodosiadis}
\author[Konstantinos Zarvalis]{K. Zarvalis}
\address{Department of Mathematics, Stockholm University, 106 91 Stockholm, Sweden}
\email{eleftherios@math.su.se}
\address{Department of Mathematics, Aristotle University of Thessaloniki, 54124, Thessaloniki, Greece}
\email{zarkonath@math.auth.gr}
\date{\today}


\begin{document}

\maketitle
\maketitle
\begin{abstract}
    We study the chordal Loewner equation associated with certain driving functions that produce infinitely many slits. Specifically, for a choice of a sequence of positive numbers $(b_n)_{n\ge1}$ and points of the real line $(k_n)_{n\ge1}$, we explicitily solve the Loewner PDE 
    $$ \dfrac{\partial f}{\partial t}(z,t)=-f'(z,t)\sum_{n=1}^{+\infty}\dfrac{2b_n}{z-k_n\sqrt{1-t}}$$
    in $\mathbb{H}\times[0,1)$. Using techniques involving the harmonic measure, we analyze the geometric behaviour of its solutions, as $t\rightarrow1^-$.
\end{abstract}
\section{Introduction}
	
Given an increasing family of slit domains $(H_t)_{0\le t<T}$ of the upper half-plane $\mathbb{H}$, that is $H_t:=\mathbb{H}\setminus\gamma([0,t])$ for a continuous curve $\gamma:[0,T]\to\overline{\mathbb{H}}$, with $\gamma(0)\in\mathbb{R}$ and $\gamma((0,T))\subset\mathbb{H}$, there exists a family of conformal maps $g_t=g(\cdot,t):H_t\overset{onto}{\longrightarrow}\mathbb{H}$, normalized in such a way that the \textit{hydrodynamic condition} $g_t(z)-z\rightarrow0$, as $z\rightarrow\infty$, is satisfied. Furthermore, an application of the Schwarz reflection principle, shows that $g_t$ has a Laurent expansion at infinity
 \begin{equation*}
     g_t(z)=z+\dfrac{b(t)}{z}+\cdots,
 \end{equation*}
 for all $z\in\mathbb{C}$ that lie outside a disc containing $\gamma([0,t])\cup\gamma([0,t])^*$, where $K^*$ denotes the reflection of $K$ with respect to the real axis. The coefficient $b(t)>0$ is called the \textit{half-plane capacity} of $\gamma([0,t])$, denoted by $\text{hcap}(\gamma([0,t]))$.
 Then, the \textit{chordal} Loewner differential equation reads for the initial value problem
 \begin{equation}\label{LE1}
    \dfrac{\partial}{\partial t}g_t(z)= \dfrac{b'(t)}{g_t(z)-\lambda(t)}, \quad g_0(z)=z,
 \end{equation}
for all $z\in H_t$ and $0\le t <T$, where $\lambda(t)$ is a continuous real-valued function of $t$, given as $\lambda(t)=g_t^{-1}(\gamma(t))$, called the \textit{driving function}. In the literature, the half-plane capacity mostly appears as $b(t)=2t$, which is possible by means of a time reparameterization. See \cite{mon} for a detailed derivation of equation (\ref{LE1}).

In the opposite direction, if we consider the initial value problem (\ref{LE1}), we then let $T_z$ be the supremum of all $t$, such that the solution to the equation is well defined and $g_t(z)\in\mathbb{H}$ for all $t\le T_z$. Then, the domain $H_t:=\{z\in\mathbb{H}: T_z>t\}$ is simply connected, $K_t:=\mathbb{H}\setminus H_t$ is compact and $\mathbb{H}\setminus K_t$ is also simply connected. Moreover, $g_t$ maps $H_t$ conformally onto $\mathbb{H}$, satisfying the hydrodynamic condition. We refer to $(K_t)_{t}$, as the \textit{compact hulls} generated by (\ref{LE1}); see \cite[Chapter 4]{lawl} for details. It is not true, in general, that an arbitrary continuous function $\lambda(t)$ produces hulls that are curves. Several authors have studied the relation between a driving function and the corresponding hulls, showing that \text{Lip}-$\frac{1}{2}$ driving functions, with sufficiently small $\text{Lip}(\frac{1}{2})$-norm, produce quasi-slit domains (see e.g. \cite{lind}, \cite{MR} and  \cite{slei1}).

Turning back to equation (\ref{LE1}), one can consider its multiple, finite or infinite, curve version. For instance, as we see in \cite{star}, given $n$ disjoint Jordan curves that emanate from the real line towards the upper half-plane, the corresponding Loewner flow is produced by the driving functions $\lambda_1,\dots,\lambda_n$. Furthermore, in \cite{tec}, the authors show that if the hulls are made up of infinitely many slits $\Gamma_n$ parameterized in $[0,1]$ such that each curve $\Gamma_j(t)$ can be separated from the closure of $\bigcup_{n\neq j}\Gamma_{j}(t)$ at each time $t$, by open sets, then the Loewner equation is written as 
 \begin{equation}\label{tecn}
     \dfrac{\partial g}{\partial t}
	 (z,t)=\sum_{n=1}^{+\infty}\dfrac{b_n(t)}{g(z,t)-\lambda_n(t)},
 \end{equation}
	for $z\in\mathbb{H}\setminus\bigcup_{n=1}^{+\infty}\Gamma_n(t)$ and a.e $0\le t\le1$, where $\sum_{n=1}^{+\infty}b_n(t)=\partial_t\textrm{hcap}(\bigcup_{n=1}^{+\infty}\Gamma_{n}(t))$.

To the authors' best knowledge, explicit solutions to the Loewner equation, which involve infinitely many slits, do not appear in the literature often. In this article, our goal is to present Loewner flows driven by infinitely many driving functions of the form $\lambda_n(t)=k_n\sqrt{1-t}$, by solving the corresponding differential equation. Single-slit versions have been studied in \cite{kad} and \cite{mar}, where the authors begin with the driving function $\lambda(t)=k\sqrt{1-t}$, and show that the geometry of the associated slit varies according to $k$. In \cite{lefteris}, the first author generalizes their result in the multi-slit version, by considering $n$ driving functions of the form $k_j\sqrt{1-t}$, to find four different geometric possibilities for the resulting slits. More specifically, they either \textit{spiral} about some point of the upper half-plane, or intersect the real line \textit{non-tangentially} (each slit by a different angle), \textit{tangentially} (all slits by the same angle, either $0$ or $\pi$) or \textit{orthogonally} (all slits by angle $\frac{\pi}{2}$). In this article, we shall extend the preceding result in the case of infinitely many slits and shall conclude that the same geometric possibilities occur.

To start with, we set our configuration by choosing a summable sequence $(b_n)_{n\ge1}$ of positive numbers and a sequence of distinct real points $(k_n)_{n\ge1}$ ordered in such a way that $\mathbb{R}$ can be written as a countable union of bounded intervals of the form $[k_m,k_{m'}]$, not containing in their interior any of the $k_n$'s, and unbounded intervals of the form $(-\infty,k_m]$ or $[k_{m'},+\infty)$, if such intervals exist, again not containing any other point $k_n$. We formally write the last condition for the sequence $(k_n)_{n\ge1}$ as

\begin{equation}\label{initial condition}
    \begin{split}
         \mathbb{R}=\overline{I_-}\cup\bigcup_{j=1}^{+\infty}\overline{I_j}\cup\overline{I_+},\hspace{1.5mm}\text{where $I_j$ are defined for all} \ j\ge1, \text{such that} \hspace{10mm}\\
         \text{there exists some}\hspace{1.5mm}j'\neq j: I_j=(k_j,k_{j'}),\hspace{1.5mm}\text{so that}\hspace{1.5mm}  k_n\notin I_j,  \forall n\ge1,\hspace{1.5mm}\hspace{8,5mm}\\
        I_-=(-\infty,\min_{n\ge1}{k_n})\hspace{1.5mm}\text{and}\hspace{1.5mm}I_+=(\max_{n\ge1}{k_n},+\infty),\hspace{1.5mm}\text{with}\hspace{1.5mm}k_n\notin I_{\pm}, \forall n\ge1.\hspace{5mm}\\
        \text{Furthermore, we assume that}\hspace{1.5mm} d:=\inf_{n\ge1}\abs{I_n}>0.\hspace{40mm}
    \end{split}
\end{equation}
Note that the left endpoint of $I_j$, determines its enumeration and note, also, that depending on the choice of the sequence $(k_n)_{n\ge1}$, either one of the two unbounded sets $I_-$ or $I_+$ might be empty. In particular, if $I_+$ is not empty, then the sequence $(k_n)_{n\ge1}$ has a maximum, say $k_N$. In this case, the interval $I_N$ as described above ceases to exist, but this omission does not impact our study.

Then, we write equation (\ref{tecn}) as
\begin{equation}\label{LE2}
    \dfrac{\partial g}{\partial t}
	 (z,t)=\sum_{n=1}^{+\infty}\dfrac{2b_n}{g(z,t)-k_n\sqrt{1-t}}
\end{equation}
	 with $g(z,0)=z$. The technique to solve this equation is straightforward and involves the transformation $\hat{g}(z,t)=\frac{g(z,t)}{\sqrt{1-t}}$, which, as we shall see later on, allows us to transform (\ref{LE2}) into the separable differential equation $\frac{d\hat{g}}{P_{\mathbb{H}}(\hat{g})}=\frac{dt}{2(1-t)}$, where
\begin{equation*}
    P_{\mathbb{H}}(z):=z+\sum_{n=1}^{+\infty}\dfrac{4b_n}{z-k_n}.
\end{equation*}

It turns out that the geometric properties of $g$ depend on the nature of the roots of the auxiliary function $P_{\mathbb{H}}$. We can already see that for each $m\in\mathbb{N}$ and for $x\in\mathbb{R}$, $\lim_{x\rightarrow k_m^+}P_{\mathbb{H}}(x)=+\infty$ and $\lim_{x\rightarrow k_m^-}P_{\mathbb{H}}(x)=-\infty$, which implies that there exists some $\lambda_{m}\in I_m$, real root of $P_{\mathbb{H}}$, so that $P'_{\mathbb{H}}(\lambda_m)<0$. For the sake of clarity, we proceed to the following definition.

\begin{define}\label{standardroots}
    We characterize any root $\rho\in\mathbb{R}$ of $P_{\mathbb{H}}$, satisfying $P_{\mathbb{H}}'(\rho)<0$, as a \textit{standard} root of $P_{\mathbb{H}}$. Via the note above, we may see that each bounded interval $I_m$ contains at least one standard root.
\end{define}
A basic part of our analysis is to decipher what other zeroes might exist. In Section 3, we uncover the properties of $P_{\mathbb{H}}$, which will allow us to determine exactly the possible additional roots, apart from the standard roots $\lambda_m$ mentioned above. For example, if we consider the case when only finitely many $b_n$'s are non-zero, then $P_{\mathbb{H}}$ is a rational function and thus its roots are determined by a polynomial of finite degree. Hence, with the use of the fundamental theorem of algebra, we can distinguish between four cases for the roots, namely one complex root (and its conjugate), or a multiple real root of degree 2 or 3, or distinct real roots; see \cite{lefteris} for details. In the case of infinitely many driving functions, the corresponding auxiliary function $P_{\mathbb{H}}$ is no longer rational and hence the aforementioned analysis no longer applies. Therefore, we come up with different strategies and we use complex analytic methods to deduce our results. It is our intention, at first, to prove that the same four cases are the only possibilities,  as they appear in the proposition below.

 \begin{proposition}
    \label{possible roots0}
    
    \textup{(1)} If there exists a complex root $\beta\in\mathbb{H}$ of $P_{\mathbb{H}}$, then $\beta, \bar{\beta}$ are the unique roots in $\mathbb{C}\setminus\mathbb{R}$, they are simple and each bounded interval $I_j$ (as described in (\ref{initial condition})) has exactly one root of $P_{\mathbb{H}}$, which is the standard root $\lambda_j$. If there exist unbounded intervals, they contain no real roots of $P_{\mathbb{H}}$.

    \textup{(2)} If some interval $I_j$ contains distinct simple real roots of $P_{\mathbb{H}}$, then $I_j$ has exactly three distinct roots, two of which are standard, if it is bounded, or two distinct roots, one of which is standard, if it is unbounded. Each other interval $I_k$ has exactly one root, which is the standard, if $I_k$ is bounded and none if $I_k$ is unbounded, and $P_{\mathbb{H}}$ has no complex roots.

    \textup{(3)} If there exists a multiple real root $\rho_0$ of $P_{\mathbb{H}}$, lying either in a bounded or in an unbounded interval (assuming that such an unbounded interval exists), then $P_{\mathbb{H}}$ has only real roots and we have the following cases:
        \begin{enumerate}
        \item[\textup{(a)}] either $\rho_0$ is a double root and each bounded interval $I_j$ has exactly one simple root, the standard root $\lambda_j$, while if there exist unbounded intervals, they contain no roots (except possibly $\rho_0$),

        \item[\textup{(b)}] or $\rho_0$ is a triple root which can only lie in some bounded interval $I_m$ and coincide with the standard root $\lambda_m$. Each other bounded interval $I_j$ has exactly one simple root, the standard root $\lambda_j$, while if there exist unbounded intervals, they contain no roots.
        \end{enumerate}
 
 \end{proposition}

We summarize by presenting the main result of this work. But before that, let us point out that it is useful to consider the inverse functions $f(\cdot,t):=f_t=g_t^{-1}:\mathbb{H}\overset{onto}{\longrightarrow} H_t$. It is then direct to see by equation (\ref{LE2}), that $f$ satisfies the PDE in $\mathbb{H}\times[0,1)$,
 
	      \begin{equation}\label{PDE1}
	          \dfrac{\partial f}{\partial t}(z,t)=-f'(z,t)\sum_{n=1}^{+\infty}\dfrac{2b_n}{z-k_n\sqrt{1-t}},
	      \end{equation}
 with initial value $f(z,0)=z$. We refer to the solution of the preceding equation as the \textit{chordal Loewner flow}. Having established all possibilities for the roots of $P_{\mathbb{H}}$ in Section 3, we then proceed to Section 4, where we explicitly solve equation $(\ref{PDE1})$ and visualize the geometry of the flow. In addition, we are going to describe its asymptotic behavior as $t\rightarrow1^-$. This behavior has been studied extensively in the past. However, in this article, we provide another approach towards the aim of tackling the problem. More specifically, we make use of a conformal invariant, the \textit{harmonic measure}. In fact, harmonic measure and its conformally invariant nature prove to be a powerful tool for studying how a trajectory behaves asymptotically.
 
 We are, now, ready to state our main result.

\begin{thm}\label{chordaltheorem}
Consider a summable sequence of positive numbers $(b_n)_{n\ge1}$ and a sequence of distinct real point $(k_n)_{n\ge1}$, satisfying condition $(\ref{initial condition})$. Then, the initial value problem $(\ref{PDE1})$ admits a unique solution in $\mathbb{H}\times[0,1)$ and we distinguish the following cases:
	      
	      \begin{enumerate}
	          \item[\textup{(1)}] If $P_{\mathbb{H}}$ has a complex root $\beta\in\mathbb{H}$, then the Loewner flow is of the form $$f(z,t)=h^{-1}\left((1-t)^{\alpha e^{-i\psi}}h\left((1-t)^{-\frac{1}{2}}z\right)\right),$$
	          where $h$ maps the upper half plane onto the complement of infinitely many logarithmic spirals of angle $-\psi$, where $\alpha$ and $\psi$ depend on the sequences $(k_n)_{n\ge1}$ and $(b_n)_{n\ge1}$.

	          \item[\textup{(2)}] If $P_{\mathbb{H}}$ has three distinct real roots, in some interval described in (\ref{initial condition}), then the Loewner flow is of the form $$f(z,t)=h^{-1}\left((1-t)^{\alpha}h\left((1-t)^{-\frac{1}{2}}z\right)\right),$$
	          where $h$ is a Schwarz-Cristoffel map of the upper half plane, that maps $\mathbb{H}$ onto $\mathbb{H}$ minus infinitely many line segments emanating from the origin.
	          \item[\textup{(3)}] If $P_{\mathbb{H}}$ has a multiple root, either double or triple, then the Loewner flow is of the form $$f(z,t)=h^{-1}\left(\dfrac{1}{2}\log(1-t)+h\left((1-t)^{-\frac{1}{2}}z\right)\right),$$
	          where $h$ is a univalent map of the upper half plane, that maps $\mathbb{H}$ onto:\\
	          \textup{(a)} either a horizontal half-plane, minus infinitely many half-lines parallel to $\mathbb{R}$, extending to the point at infinity from the left, if the root is double,\\
	          \textup{(b)} or the complement of infinitely many half-lines parallel to $\mathbb{R}$, extending to the point at infinity from the left, if the root is triple.  
	      \end{enumerate}
       
            Furthermore, for all $n\ge1$,  we have that the trajectories of the driving functions $\hat{\gamma}^{(n)}:=\{f(k_n\sqrt{1-t},t):\ t\in[0,1)\}$, are smooth curves of $\mathbb{H}$ starting at $k_n$, that spiral about the point $\beta$ in case $(1)$, intersect $\mathbb{R}$ at one of the real roots non-tangentially in case $(2)$ and intersect $\mathbb{R}$ at the multiple root tangentially in case $(3a)$ or orthogonally in case $(3b)$.	 
\end{thm}

 \begin{figure}[h]
    \centering
    \includegraphics[scale=0.55]{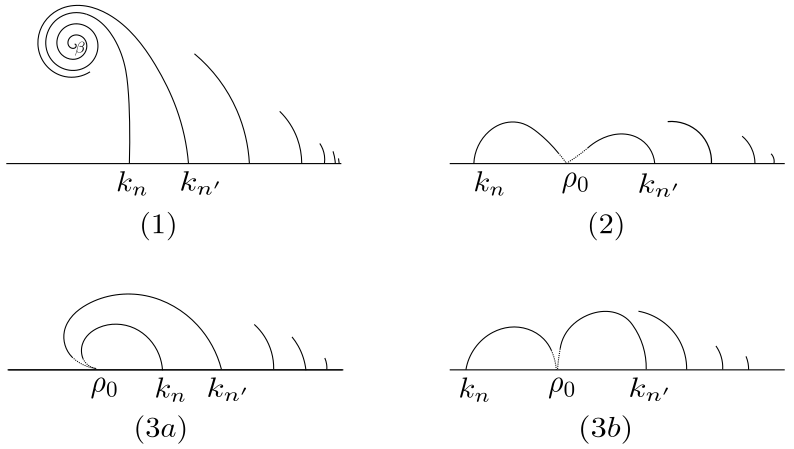}
    \caption{The geometric behaviour of the tip points $f(k_n\sqrt{1-t},t)$ for each case of Theorem \ref{chordaltheorem}.}
    \label{fig:cases}
\end{figure}

The structure of the article is as follows: In Section 2, we collect the preliminary results that will be needed during the course of the proofs. Then, in Section 3, we will proceed to a complete study of the function $P_\mathbb{H}$ and its nature by examining its roots. Finally, in Section 4, we shall state and prove a series of lemmas and propositions, whose combination will lead to Theorem \ref{chordaltheorem}.

\section{Basic tools and preliminaries}

\subsection{Theory of conformal mappings}

We begin by presenting some basic definitions about spirallike domains, that will be necessary for later. A \textit{logarithmic spiral of angle} $\psi\in(-\frac{\pi}{2},\frac{\pi}{2})$ in the complex plane, that joins the origin with infinity and passes from some point $w_0\neq0$, is defined as the curve with parameterization $S: w=w_0 \text{exp}(-e^{i\psi}t)$, $-\infty\le t\le+\infty$.
\begin{define}
\textup{(1)} A simply connected domain $D$, that contains the origin, is said to be $\psi$\textit{-spirallike (with respect to $0$)}, if for any point $w_0\in D$, the logarithmic spiral $S: w=w_0\exp(-e^{i\psi}t)$, $0\le t\le+\infty$, is contained in $D$.

\textup{(2)} A  univalent function $f\in H(\mathbb{H})$, with $f(\beta)=0$ for some $\beta\in\mathbb{H}$, is said to be $\psi$\textit{-spirallike (with respect to $\beta$)}, if it maps the upper half plane onto a $\psi$-spirallike domain $D$ (with respect to 0).
\end{define}

For more details on spirallike functions, the interested reader may refer to \cite[\S2.7]{dur}. Note that $0$-spirals are half-lines emanating from the origin and extending to infinity. We refer to $0$-spirallike domains/functions as \textit{starlike} domains/functions. The following theorem gives an analytic characterization of spirallike mappings. 

\begin{thm} \label{spirallikeinH}
     Let $f\in H(\mathbb{H})$, with $f'(\beta)\neq0$ and $f(z)=0$ if and only if $z=\beta$. Then, $f$ is $\psi$-spirallike, if and only if 
$$\text{Im}\left(e^{-i\psi}\dfrac{(z-\beta)(z-\bar{\beta})f'(z)}{f(z)}\right)>0,$$
for all $z\in\mathbb{H}$.   
\end{thm}
\begin{proof}
    By \cite[Theorem 2.19]{dur}, a function $g\in H(\mathbb{D})$, with $g'(0)\neq0$ and $g(z)=0$ if and only if $z=0$, is $\psi$-spirallike, if and only if 
    $$\mathrm{Re}\left(e^{-i\psi}\dfrac{zg'(z)}{g(z)}\right)>0,$$
    for all $z\in\mathbb{D}$. Considering the M\"{o}bius transform $T(z)=\frac{z-\beta}{z-\bar{\beta}}:\mathbb{H}\rightarrow\mathbb{D}$ and applying the preceding result to the function $g:=f\circ T^{-1}$, the desired inequality follows.
\end{proof}

Aside from that and returning to our setting, as we mentioned in the introduction, when $b_n$ is non-zero for finitely many $n\ge1$, $P_{\mathbb{H}}$ is a  rational function. As we shall see in Section 4, in order to solve PDE $(\ref{PDE1})$, we need to apply partial fraction decomposition on $\frac{1}{P_{\mathbb{H}}}$, which can be done since this is a rational function as well. However, the same method is not applicable when $b_n>0$, for all $n\ge1$. For this reason, we are going to require the following theorem, which is an application of the residue calculus (see \cite[Section II.9]{Mark}) for details.
\begin{thm}{\cite[Theorem II.2.7]{Mark}}\label{partial fractions expansion}
    Let $f$ be an analytic function in $\mathbb{C}$, whose only singularities are poles in the finite plane, say $(\lambda_n)_{n\ge1}$, and let $G_n$ be the principal part of $f$ at each $\lambda_n$, respectively. Let $(L_n)_{n\ge1}$ be a sequence of contours which satisfy the following:
    \begin{enumerate}
    \item[\textup{(i)}] each contour $L_n$ contains finitely many poles, 
    \item[\textup{(ii)}] $0\in L_n\subset\text{Int}(L_{n+1})$ for all $n\ge1$,
    \item[\textup{(iii)}] $r_n:=\text{dist}(0,L_n)\rightarrow +\infty$.
    \end{enumerate}
    If $\limsup_{n\rightarrow\infty}\int_{L_n}|f(\zeta)|d\zeta<\infty$,
    then $$f(z)=\sum_{n=1}^{+\infty}G_n(z)$$
    and the convergence is uniform on compacta.
\end{thm}

\subsection{Harmonic Measure}

During the course of several proofs, we are going to utilize one conformal invariant, the \textit{harmonic measure}. An introductory presentation of its rich theory may be found in \cite{harmonic}. For the purposes of the present article, we will only review some basic facts. Let $\Omega\subsetneq\mathbb{C}$ be a domain with non-polar boundary. Let $E$ be a Borel subset of $\partial\Omega$. Then, the harmonic measure of $E$ with respect to $\Omega$ is exactly the solution of the generalized Dirichlet problem for the Laplacian in $\Omega$ with boundary function equal to $1$ on $E$ and to $0$ on $\partial\Omega\setminus E$. For the harmonic measure of $E$ with respect to $\Omega$ and for $z\in\Omega$, we use the notation $\omega(z,E,\Omega)$. It is known that for a fixed $z\in\Omega$, $\omega(z,\cdot,\Omega)$ is a Borel probability measure on $\partial\Omega$.

As we already mentioned, the harmonic measure is conformally invariant. In addition, it has a very useful monotonicity property. More specifically, let $\Omega_1\subset\Omega_2\subsetneq\mathbb{C}$ be two domains with non-polar boundaries and let $E\subset\partial\Omega_1\cap\partial\Omega_2$. Then,
$$\omega(z,E,\Omega_1)\le\omega(z,E,\Omega_2), \quad\text{for all }z\in\Omega_1.$$

Furthermore, later on we will need some formulas for harmonic measure in certain particular domains (see \cite[p.100]{harmonic}). Let $[a,b]\subset\mathbb{R}$ and let $z\in\mathbb{H}$. Then, 
$$\omega(z,[a,b],\mathbb{H})=\frac{1}{\pi}\arg\left(\frac{z-b}{z-a}\right).$$
As a consequence, given any curve $\gamma:[0,+\infty)\to\mathbb{H}$ satisfying $\lim_{t\to+\infty}\gamma(t)=\infty$, it can be easily calculated that $\lim_{t\to+\infty}\omega(\gamma(t),[a,b],\mathbb{H})=0$.

Finally, for the angular domain $U_{\alpha,\beta}:=\{w\in\mathbb{C}:\alpha<\arg w<\beta\}$, where $0\le\alpha<\beta<2\pi$ (or $-\pi\le\alpha<\beta<\pi$), and for $z\in U_{\alpha,\beta}$, we have
$$\omega(z,\{w\in\mathbb{C}:\arg w=\alpha\},U_{\alpha,\beta})=\frac{\beta-\arg z}{\beta-\alpha}=1-\omega(z,\{w\in\mathbb{C}:\arg w=\beta\},U_{\alpha,\beta}).$$

\begin{rmk}\label{cara}
    Suppose that $E\subset\partial\mathbb{D}$, where $\mathbb{D}$ is the unit disk, is a circular arc with endpoints $a,b$. Then, we know (see e.g. \cite[p.155]{carath}) that the level set
    $$D_k=\left\{z\in\mathbb{D}:\omega(z,E,\mathbb{D})=k\right\}, \quad k\in(0,1),$$
    is a circular arc (or a diameter in case $k=\frac{1}{2}$ and $E$ is a half-circle) inside $\mathbb{D}$ with endpoints $a,b$ that intersects $\partial\mathbb{D}$ with angles $k\pi$ and $(1-k)\pi$.
\end{rmk}

\begin{rmk}
    An important piece of information with regard to harmonic measure is its probabilistic interpretation. In particular, the quantity $\omega(z,E,\Omega)$ represents the probability of a Brownian motion starting at $z$ to exit the domain $\Omega$ for the first time passing through $E$. This thought, combined with the previous remark, justifies intuitively our usage of harmonic measure in the study of angles.
\end{rmk}

\begin{rmk}
    In many cases, instead of Borel sets, we use sets of prime ends for the harmonic measure. This is possible through Carath\'{e}odory's Theorem concerning the extension of the Riemann mapping theorem to the boundary (see \cite[Chapter 9]{pom} for details on prime ends and the boundary behavior of conformal mappings). To be more exact, suppose that $\Omega\subsetneq\mathbb{C}$ is a simply connected domain with non-polar boundary and $f:\mathbb{D}\to\Omega$ is a corresponding Riemann mapping. Suppose that $E\subset\partial\mathbb{D}$ is Borel. Then, this set $E$ corresponds through $f$ to a set of prime ends $\hat{f(E)}$ of $\Omega$ (if $\Omega$ is not a Jordan domain, then there might not be an one-to-one correspondence between $\hat{f(E)}$ and points of $\partial\Omega$ that are limit points of $\lim\limits_{z\to \zeta}f(z)$, $\zeta\in E$). Therefore, by the conformal invariance of the harmonic measure and Carath\'{e}odory's Theorem, we may write
    $$\omega(z,E,\mathbb{D})=\omega(f(z),\hat{f(E)},\Omega), \quad z\in\mathbb{D}.$$
\end{rmk}

\section{Properties of the driving function}
Given a summable sequence of non-negative numbers $(b_n)_{n\ge1}$ and a sequence of real points $(k_n)_{n\ge1}$, the function 
$$P_{\mathbb{H}}(z):=z+\sum_{n=1}^{+\infty}\dfrac{4b_n}{z-k_n}$$
is quickly verified to converge locally uniformly in $\mathbb{C}\setminus\overline{(k_n)_{n\ge1}}$. As we mentioned previously, the key role with regard to the geometry of the slits is played by the roots of $P_\mathbb{H}$. If the non-zero terms of the sequence $(b_n)_{n\ge1}$ are finitely many, then $P_{\mathbb{H}}$ is a rational function. Therefore, we are able to use the fundamental theorem of algebra in order to count its roots and then distinguish the cases when it has complex roots or only real roots. But since we do not have this theorem at our disposal in this case, we need to come up with different techinques. We start with the following lemma.

\begin{lemma}\label{FGH}
    Given the parameters above we have the following expressions:
    
\textup{(1)} If $F(z,w):=\dfrac{P_{\mathbb{H}}(z)}{(z-w)(z-\bar{w})}$, then

        $$F(z,w)=\left(\dfrac{P_{\mathbb{H}}(w)}{z-w}-\dfrac{P_{\mathbb{H}}(\bar{w})}{z-\bar{w}}\right)\dfrac{1}{w-\bar{w}}+\sum_{n=1}^{+\infty}\dfrac{4b_n}{(z-k_n)|w-k_n|^2},$$
        for all $z,w\in\mathbb{C}\setminus(k_n)_{n\ge1}$ with $z\neq w$.

       \textup{(2)} If $H(z,\lambda_1,\lambda_2):=\dfrac{P_{\mathbb{H}}(z)}{(z-\lambda_1)(z-\lambda_2)}$, then
       $$H(z,\lambda_1,\lambda_2)=\left(\frac{P_{\mathbb{H}}(\lambda_1)}{z-\lambda_1}-\frac{P_{\mathbb{H}}(\lambda_2)}{z-\lambda_2}\right)\dfrac{1}{\lambda_1-\lambda_2}+\sum_{n=1}^{+\infty}\dfrac{4b_n}{(z-k_n)(\lambda_1-k_n)(\lambda_2-k_n)},$$
        for all $z,\lambda_1,\lambda_2\in\mathbb{C}\setminus(k_n)_{n\ge1}$ with $z\neq\lambda_1,\lambda_2$ and $\lambda_1\neq\lambda_2$.

         \textup{(3)} If $G(z,\lambda):=\dfrac{P_{\mathbb{H}}(z)}{(z-\lambda)^2}$, then
          $$G(z,\lambda)=\dfrac{P_{\mathbb{H}}'(\lambda)}{z-\lambda}+\dfrac{P_{\mathbb{H}}(\lambda)}{(z-\lambda)^2}+\sum_{n=1}^{+\infty}\dfrac{4b_n}{(z-k_n)(\lambda-k_n)^2},$$
        for all $z,\lambda\in\mathbb{C}\setminus(k_n)_{n\ge1}$ with $z\neq\lambda$.
  
\end{lemma}
\begin{proof}
    We commence by proving relation (2). By executing consecutive partial fraction decompositions, we have that 
    \begin{align*}
        &H(z,\lambda_1,\lambda_2)=\dfrac{z}{(z-\lambda_1)(z-\lambda_2)}+\sum_{n=1}^{+\infty}\dfrac{\frac{4b_n}{z-k_n}}{(z-\lambda_1)(z-\lambda_2)}\\
        &=\dfrac{1}{\lambda_1-\lambda_2}\left(\dfrac{\lambda_1}{z-\lambda_1}-\dfrac{\lambda_2}{z-\lambda_2}+\sum_{n=1}^{+\infty}\dfrac{4b_n}{(z-k_n)(z-\lambda_1)}-\sum_{n=1}^{+\infty}\dfrac{4b_n}{(z-k_n)(z-\lambda_2)}\right)\\
       &=\dfrac{1}{\lambda_1-\lambda_2}\left[\dfrac{\lambda_1}{z-\lambda_1}-\dfrac{\lambda_2}{z-\lambda_2}+\sum_{n=1}^{+\infty}\left(\dfrac{\frac{4b_n}{\lambda_1-k_n}}{z-\lambda_1}-\dfrac{\frac{4b_n}{\lambda_1-k_n}}{z-k_n}\right)-\sum_{n=1}^{+\infty}\left(\dfrac{\frac{4b_n}{\lambda_2-k_n}}{z-\lambda_2}-\dfrac{\frac{4b_n}{\lambda_2-k_n}}{z-k_n}\right)\right]\\
       &=\dfrac{1}{\lambda_1-\lambda_2}\left(\dfrac{\lambda_1+\sum_{n=1}^{+\infty}\frac{4b_n}{\lambda_1-k_n}}{z-\lambda_1}-\dfrac{\lambda_2+\sum_{n=1}^{+\infty}\frac{4b_n}{\lambda_2-k_n}}{z-\lambda_2}\right)+\dfrac{1}{\lambda_1-\lambda_2}\sum_{n=1}^{+\infty}\dfrac{\frac{4b_n}{\lambda_2-k_n}-\frac{4b_n}{\lambda_1-k_n}}{z-k_n}
    \end{align*}
    and the third statement follows. Notice that $F(z,w)=H(z,w,\bar{w})$ and $G(z,\lambda)=H(z,\lambda,\lambda)$. Then, relations (1) and (3) follow.
\end{proof}

\begin{define}
    An interval of the form $I_n=(k_n,k_{n'})$ that does not contain any point of the sequence $(k_n)_{n\ge1}$ is called \textit{bounded interval} of $P_{\mathbb{H}}$. 
     If there exists an interval of the form $I_-=(-\infty,k_{n'})$ or $I_+=(k_{n'},+\infty)$ that does not contain any point of the sequence $(k_n)_{n\ge1}$, then it is called \textit{left or right unbounded interval} of $P_{\mathbb{H}}$, respectively. 
\end{define}

The preceding proposition is important, as it reveals the mechanism, according to which, $P_{\mathbb{H}}$ can have a unique complex root, or a unique multiple real root, or there exists a unique interval of $P_{\mathbb{H}}$ (bounded or unbounded) that contains three distinct real roots. Indeed, as we shall see in a subsequent corollary, the left summands in the relations $(1)-(3)$ above vanish, when $w, \lambda, \lambda_j$ are roots of $P_{\mathbb{H}}$ and as a result the imaginary parts of $F,G$ and $H$ are negative and positive in the upper half plane and in the lower half plane, respectively, allowing us to deduce uniqueness.

   \begin{lemma}\label{roots in intervals}
       A bounded interval of $P_{\mathbb{H}}$ contains either one or three roots of $P_{\mathbb{H}}$, counting multiplicity.
       An unbounded interval of $P_{\mathbb{H}}$ contains either none or two roots of $P_{\mathbb{H}}$, counting multiplicity.
   \end{lemma} 
   \begin{proof}
      In a bounded interval of $P_{\mathbb{H}}$, say $(k_n,k_{n'})$, we observe that $\lim_{\mathbb{R}\ni x\to k_n^+}P_{\mathbb{H}}(x)=+\infty$ and $\lim_{\mathbb{R}\ni x\to {k_{n'}^-}}P_{\mathbb{H}}(x)=-\infty$. Thus, there exists some $\lambda_{n}\in(k_n,k_{n'})$, root of $P_{\mathbb{H}}$. Because the third derivative of $P_{\mathbb{H}}$ is always negative in $\mathbb{R}$, by Rolle's theorem, we can have, counting multiplicity, up to three real roots in $(k_n,k_{n'})$. In particular, if $P_\mathbb{H}$ has exactly two distinct roots, then one of them is necessarily a double root. 
      
      Besides, in case there exists a left unbounded interval, then we may compute that $\lim_{x\to-\infty}P_\mathbb{H}(x)=-\infty$, whereas if there exists a right unbounded interval, then $\lim_{x\to+\infty}P_\mathbb{H}(x)=+\infty$. Using similar arguments as above, we may deduce the desired result.
   \end{proof}
   As a direct corollary we may now prove Proposition \ref{possible roots0}.
\begin{proof}[Proof of Proposition \ref{possible roots0}]
 
        (1) Assume that $\beta\in\mathbb{H}$ is a complex root of $P_{\mathbb{H}}$. Since the parameters $b_n$ and $k_n$ are real, $\bar{\beta}$ is also a root. By Lemma \ref{FGH}, the imaginary part of $F(z,\beta)$ is negative for all $z\in\mathbb{H}$ and positive for all $z\in-\mathbb{H}$. This implies that $\beta,\bar{\beta}$ are simple roots of $P_{\mathbb{H}}$ and they are the unique non-real roots. Finally, we have that $F'(x,\beta)<0$, for all $x\in\mathbb{R}\setminus(k_n)_{n\ge1}$, which means that $F(\cdot,\beta)$ is decreasing in each interval of  $P_{\mathbb{H}}$ and the result follows.
        
        (2) Assume that there exist two distinct roots $\lambda_1,\lambda_2$ in some interval of $P_{\mathbb{H}}$. Again, we have that $\text{Im}(H(z,\lambda_1,\lambda_2))\neq0$, for all $z\in\mathbb{C}\setminus\mathbb{R}$. Hence, $H$ and by extension $P_\mathbb{H}$ cannot have non-real roots. Since $\lambda_1,\lambda_2$ lie in the same interval, we have that $(\lambda_1-k_n)(\lambda_2-k_n)>0$, for all $n\in\mathbb{N}$. Therefore, $H'(x,\lambda_1,\lambda_2)<0$, for all $x\in\mathbb{R}\setminus(k_n)_{n\ge1}$ and this completes the proof.

         (3) By Lemma \ref{FGH}, the imaginary part of $G(z,\rho_0)$ is negative for all $z\in\mathbb{H}$ and positive for all $z\in-\mathbb{H}$. Therefore, $G$ and $P_{\mathbb{H}}$ have no roots in $\mathbb{C}\setminus\mathbb{R}$. As before, we have that $G'(x,\rho_0)<0$ for all $x\in\mathbb{R}\setminus(k_n)_{n\ge1}$, which means that $G(\cdot,\rho)$ is decreasing in each interval of  $P_{\mathbb{H}}$. This monotony implies the desired result.

\end{proof}
We recall at this point that by condition $(\ref{initial condition})$, the terms of $(k_n)_{n\ge1}$ are chosen in such a way that $d:=\inf_{n\neq m}|k_n-k_m|>0$. This allows for the very useful property that for any $\epsilon\le\frac{d}{4}$, $P_{\mathbb{H}}$ is convergent uniformly in the domain $\mathbb{C}\setminus\bigcup_{n\ge1}D(k_n,\epsilon)$. Under this assumption, we can prove the following result.
\begin{proposition}\label{possibilities}
    Consider a sequence of non-negative numbers $(b_n)_{n\ge1}$ and a sequence of real points $(k_n)_{n\ge1}$ satisfying condition $(\ref{initial condition})$ and denote by $(\lambda_n)_{n\ge1}$ the standard roots of Proposition \ref{possible roots0}. We distinguish the following cases in accordance with Proposition \ref{possible roots0}:
    
  \textup{(1)} Let $\beta\in\mathbb{H}$ be a complex root of $P_{\mathbb{H}}$ (along with its conjugate $\bar{\beta}$) and define $\psi:=\mathrm{Arg}(P_{\mathbb{H}}'(\beta))$. Then, 
        \begin{equation}\label{complex root}
\dfrac{1}{P_{\mathbb{H}}(z)}=\dfrac{\frac{1}{P_{\mathbb{H}}'(\beta)}}{z-\beta}+\dfrac{\frac{1}{P_{\mathbb{H}}'(\bar{\beta})}}{z-\bar{\beta}}+\sum_{n=1}^{+\infty}\dfrac{\frac{1}{P_{\mathbb{H}}'(\lambda_n)}}{z-\lambda_n}.
    \end{equation}
    Furthermore, we have that $\psi\in(-\frac{\pi}{2},\frac{\pi}{2})$ and
    $$\sum_{n=1}^{+\infty}\frac{1}{P_{\mathbb{H}}'(\lambda_n)}+\dfrac{2\cos(\psi)}{|P_{\mathbb{H}}'(\beta)|}=1.$$

    \textup{(2)} Assume that $P_{\mathbb{H}}$ has only distinct real roots in some interval of $P_{\mathbb{H}}$, either three in a bounded one, say $\lambda_j<\rho_1<\rho_2$, or two in an unbounded, say $\rho_1,\rho_2$. Then,
    \begin{equation}\label{distinct roots}
\dfrac{1}{P_{\mathbb{H}}(z)}=\dfrac{\frac{1}{P_{\mathbb{H}}'(\rho_1)}}{z-\rho_1}+\dfrac{\frac{1}{P_{\mathbb{H}}'(\rho_2)}}{z-\rho_2}+\sum_{n=1}^{+\infty}\dfrac{\frac{1}{P_{\mathbb{H}}'(\lambda_n)}}{z-\lambda_n}.
\end{equation}
 Furthermore,
    $$\sum_{n=1}^{+\infty}\frac{1}{P_{\mathbb{H}}'(\lambda_n)}+\dfrac{1}{P_{\mathbb{H}}'(\rho_1)}+\dfrac{1}{P_{\mathbb{H}}'(\rho_2)}=1.$$

   \textup{(3a)} Let $\rho_0\in\mathbb{R}$ be a real double root of $P_{\mathbb{H}}$. Then,

      \begin{equation}\label{double root}
\dfrac{1}{P_{\mathbb{H}}(z)}=\dfrac{\frac{2}{P_{\mathbb{H}}^{(2)}(\rho_0)}}{(z-\rho_0)^2}+\dfrac{-\frac{2P_{\mathbb{H}}^{(3)}(\rho_0)}{3(P_{\mathbb{H}}^{(2)}(\rho_0))^2}}{z-\rho_0}+\sum_{n=1}^{+\infty}\dfrac{\frac{1}{P_{\mathbb{H}}'(\lambda_n)}}{z-\lambda_n}.
\end{equation}
Furthermore,
$$ -\dfrac{2P_{\mathbb{H}}^{(3)}(\rho_0)}{3\left(P_{\mathbb{H}}^{(2)}(\rho_0)\right)^2}+\sum_{n=1}^{+\infty}\dfrac{1}{P_{\mathbb{H}}'(\lambda_n)}=1$$

\textup{(3b)} Let $\rho_0\in\mathbb{R}$ be a real triple root of $P_{\mathbb{H}}$. Then, $\rho_0=\lambda_m$ for some $m\in\mathbb{N}$ and

\begin{equation}\label{triple root}
\dfrac{1}{P_{\mathbb{H}}(z)}=\dfrac{\frac{6}{P_{\mathbb{H}}^{(3)}(\rho_0)}}{(z-\rho_0)^3}+\dfrac{-\frac{3P_{\mathbb{H}}^{(4)}(\rho_0)}{2\left(P_{\mathbb{H}}^{(3)}(\rho_0)\right)^2}}{(z-\rho_0)^2}+2\dfrac{\left(\frac{P_{\mathbb{H}}^{(4)}(\rho_0)}{4!}\right)^2-\frac{P_{\mathbb{H}}^{(3)}(\rho_0)P_{\mathbb{H}}^{(5)}(\rho_0)}{3!5!}}{\left(\frac{P_{\mathbb{H}}^{(3)}(\rho_0)}{3!}\right)^3(z-\rho_0)}+\sum_{n\neq m}\dfrac{\frac{1}{P_{\mathbb{H}}'(\lambda_n)}}{z-\lambda_n}.
\end{equation}
Furthermore,
$$2\dfrac{\left(\frac{P_{\mathbb{H}}^{(4)}(\rho_0)}{4!}\right)^2-\frac{P_{\mathbb{H}}^{(3)}(\rho_0)P_{\mathbb{H}}^{(5)}(\rho_0)}{3!5!}}{\left(\frac{P_{\mathbb{H}}^{(3)}(\rho_0)}{3!}\right)^3}+\sum_{n\neq m}\frac{1}{P_{\mathbb{H}}'(\lambda_n)}=1.$$

\end{proposition}

\begin{proof}

    We first establish the representation formulas $(\ref{complex root})-(\ref{triple root})$. Assume, without loss of generality, that the sequence $(k_n)_{n\ge1}$ accumulates at both $+\infty$ and $-\infty$. So, we can extract a subsequence $(k_{m^1_n})_{n\ge1}$ that increases to $+\infty$. Let $R_n$ be the middle points of the intervals $I_{m^1_n}$ (see condition $(\ref{initial condition})$) and observe that $(R_n)_{n\ge1}$ is increasing to $+\infty$. We, now, construct a sequence of rectangles $(L_n)_{n\ge1}$ in the following manner: first we consider $L_n$ to be the square of center $0$ with sides parallel to the axes, passing from the point $R_n$. If $-R_n$ happens to be the middle point of an interval $(k_{m_{n_2}^2},k_{m_{n_1}^2})$, we leave $L_n$ as is. If not, we transform $L_n$ into a rectangle passing from the points $R_n,\pm iR_n$ and $-R_n-\epsilon_n$ (or $-R_n+\epsilon_n$), where $\epsilon_n<\frac{d}{2}$ is picked in such a way that $|\zeta-k_m|\ge\frac{d}{2}$, for all $\zeta\in L_n$ and every $m\ge1$. We then have that $L_n\subset\mathrm{Int}(L_{n+1})$ and the perimeter of $L_n$ is less than $8R_n+2\epsilon_n$, for all $n\ge1$. We may assume that for each case described above, the additional roots $\beta,\bar{\beta}$ or $\rho_j$, lie in $\mathrm{Int}(L_1)$ and $R_1>\sigma:=\frac{2}{d}\sum_{n=1}^{+\infty}4b_n$. We then have that $|P_{\mathbb{H}}(\zeta)|=\abs{\zeta+\sum_{m=1}^{+\infty}\frac{4b_m}{\zeta-k_m}}\ge R_n-\frac{d}{2}-\sigma$, for all $\zeta\in L_n$. This implies that $\limsup_{n\rightarrow+\infty}\int_{L_n}\frac{1}{|P_{\mathbb{H}}(\zeta)|}|d\zeta|<+\infty$. For every $n\ge1$, recall that $\lambda_n$ is a simple root of $P_{\mathbb{H}}$, hence a simple pole $\frac{1}{P_{\mathbb{H}}}$ with residue $\frac{1}{P_{\mathbb{H}}'(\lambda_n)}$. For the cases (1) and (2), the roots $\beta, \bar{\beta}$ and $\rho_j$, $j=1,2$, are also simple poles of $\frac{1}{P_{\mathbb{H}}}$, while in cases (3a) and (3b), the root $\rho_0$ is a pole of second or third order, respectively. Therefore, writing down the principal part of $\frac{1}{P_{\mathbb{H}}}$ around each pole, we apply Theorem \ref{partial fractions expansion} and we conclude relations $(\ref{complex root})-(\ref{triple root})$. The convergence is locally uniform.
    
Next, we shall establish the second part of (1). Using Lemma \ref{FGH}, we observe that $\mathrm{Im}F(\beta,\beta)=\mathrm{Im}\frac{|P_{\mathbb{H}}'(\beta)|e^{i\psi}}{2i\mathrm{Im}\beta}<0$ and hence $\cos(\psi)>0$. To prove the last statement, we have that 
        $$\dfrac{iy}{P_{\mathbb{H}}(iy)}=\dfrac{1}{P_{\mathbb{H}}'(\beta)}\dfrac{iy}{iy-\beta}+\dfrac{1}{P_{\mathbb{H}}'(\bar{\beta})}\dfrac{iy}{iy-\bar{\beta}}+\sum_{n=1}^{+\infty}\dfrac{1}{P_{\mathbb{H}}'(\lambda_n)}\dfrac{iy}{iy-\lambda_n},$$
        for all $y>0$. At first, we need to show that the $(1/P'_{\mathbb{H}}(\lambda_n))_{n\ge1}$ is summable. Taking the real parts in the preceding relation, we have that 
        $$\sum_{n=1}^{+\infty}\dfrac{1}{|P_{\mathbb{H}}'(\lambda_n)|}\dfrac{y^2}{y^2+\lambda_n^2}=\text{Re}\left(\dfrac{1}{P_{\mathbb{H}}'(\beta)}\dfrac{iy}{iy-\beta}+\dfrac{1}{P_{\mathbb{H}}'(\bar{\beta})}\dfrac{iy}{iy-\bar{\beta}}-\dfrac{iy}{P_{\mathbb{H}}(iy)}\right)=:\Phi(y)>0$$
        for all $y\ge0$. Aiming for a contradiction, we assume that $\sum_{n=1}^{+\infty}\frac{1}{|P'_{\mathbb{H}}(\lambda_n)|}=+\infty$, we then consider some natural $N\ge1$, such that its $N$-th partial sum is larger than $M:=2\max_{y\ge0}\Phi(y)<\infty$. Then, we choose $y_0>0$, so that $\frac{y_0^2}{y_0^2+\lambda_n^2}\ge\frac{1}{2}$, for all $n=1,\dots,N$. As a result, 
        $$\Phi(y_0)\ge\sum_{n=1}^{N}\dfrac{1}{|P_{\mathbb{H}}'(\lambda_n)|}\dfrac{y_0^2}{y_0^2+\lambda_n^2}\ge\dfrac{1}{2}\sum_{n=1}^{N}\dfrac{1}{|P_{\mathbb{H}}'(\lambda_n)|}>\dfrac{M}{2}\ge\Phi(y_0).$$
        This contradiction shows that the series is convergent, and finally, letting $y\rightarrow+\infty$, by the uniform convergence of the series, we deduce the second statement of $(1)$. The other cases follow similarly, so we omit their proof.
\end{proof}

In the last proposition, we made four different assumptions regarding the roots of $P_{\mathbb{H}}$. Obviously, each of these cases occurs depending on the choice of the sequences $(b_n)_{n\ge1}$ and $(k_n)_{n\ge1}$. The following proposition states that for every such choice, one of these four cases can turn up and they are mutually exclusive, thus, exactly one of the four holds.
\begin{proposition}\label{allroots}
Given a sequence of real points $(k_n)_{n\ge1}$ and a summable sequence of positive numbers $(b_n)_{n\ge1}$, then only one of the possible cases of Proposition \ref{possible roots0} can occur. Therefore, either $P_{\mathbb{H}}$ has a complex root or one of the intervals of $P_{\mathbb{H}}$ has three roots counting multiplicity, if the interval is bounded, or two roots counting multiplicity if it is unbounded.
\end{proposition}
\begin{proof}
    
Assume that $P_{\mathbb{H}}$ has no complex roots. We will prove that there exists an interval (bounded or unbounded) that contains three or two roots counting multiplicity. Assume on the contrary that this is not true and hence each bounded interval $I_n$ contains only the standard root $\lambda_n$ and each unbounded (if it exists) contains no roots. Then, as in the proof of Proposition \ref{possibilities}, we have that $\frac{1}{P_{\mathbb{H}}(z)}=\sum_{n=1}^{+\infty}\frac{1}{P_{\mathbb{H}}'(\lambda_n)}\frac{1}{z-\lambda_n}$, which implies that
$$1=\lim_{y\rightarrow+\infty}\frac{iy}{P_{\mathbb{H}}(iy)}=\sum_{n=1}^{+\infty}\frac{1}{P_{\mathbb{H}}'(\lambda_n)}<0,$$
a contradiction. Therefore, there exists some interval that contains more than one root and by Proposition \ref{possible roots0}, this is the unique interval that will contain three roots, if it is bounded, or two if it is unbounded, counting multiplicity.

On the other hand, if $P_{\mathbb{H}}$ has a complex root, then by Proposition \ref{possible roots0}, this root along with its conjugate are the sole complex roots and each bounded interval of $P_{\mathbb{H}}$ has exactly one real root, whereas if there exists an unbounded one, it has no real roots.

\end{proof}

\section{The chordal Loewner flow}
Recall, now, that our purpose is to solve PDE (\ref{PDE1}) and describe its solutions geometrically. For this, we consider the corresponding ODE
\begin{equation}
    \label{ODE1}
    \dfrac{dw}{dt}(z,t)=\sum_{n=1}^{+\infty}\dfrac{2b_n}{w(z,t)-k_n\sqrt{1-t}}
\end{equation}
with initial value $w(z,0)=z$. This becomes, using the transform $v=(1-t)^{-\frac{1}{2}}w$,
\begin{equation}
    \label{ODE2}
    \dfrac{dv}{dt}(z,t)=\dfrac{1}{2(1-t)}\left(v(z,t)+\sum_{n=1}^{+\infty}\dfrac{4b_n}{v(z,t)-k_n}\right)=\dfrac{P_{\mathbb{H}}(v(z,t))}{2(1-t)}.
\end{equation}
This is a separable differential equation and therefore, the initial value problem $(\ref{PDE1})$ can be solved by integrating the equation
\begin{equation}\label{differential}
    \dfrac{1}{P_{\mathbb{H}}(v)}dv=\dfrac{1}{2(1-t)}dt.
\end{equation}
For this reason, fix some $z_0\in\mathbb{H}$ and consider the function $\phi(z):=\int_{z_0}^{z}\frac{d\zeta}{P_{\mathbb{H}}(\zeta)}$, which is well defined and analytic, since $\mathbb{H}$ is simply connected. The function $\phi$ is going to play the role of the primitive of $\frac{1}{P_{\mathbb{H}}}$. Looking at relations $(\ref{complex root})-(\ref{triple root})$, the following lemma will be necessary for later.

\begin{lemma}\label{analyticityofh}
    Let $(A_n)_{n\ge1}$ be an absolutely summable sequence of complex numbers and let $(\lambda_{n})_{n\ge1}$ be a sequence of real points. Fix some $z_0\in\mathbb{H}$. Then, the function
    $$\psi(z):=\sum_{n=1}^{+\infty}A_n\log\dfrac{z-\lambda_{n}}{z_0-\lambda_n}$$
    is analytic in $\mathbb{H}$ and continuous in $\overline{\mathbb{H}}\setminus\overline{(\lambda_n)_{n\ge1}}$.
\end{lemma}
\begin{proof}
    It suffices to prove that the series converges uniformly in a compact set $K\subset\overline{\mathbb{H}}\setminus\overline{(\lambda_n)_{n\ge1}}$. Let $R>\abs{z_0}$  such that $K\subset D(0,R)$ and let $d:=\text{dist}(K,\overline{(\lambda_n)_{n\ge1}})$. Clearly $d>0$. For all those $n\ge1$ such that $\abs{\lambda_n}\le R$ and for $z\in K$, we have that
    $$\dfrac{d}{\abs{z_0}+R}\le\left|\dfrac{z-\lambda_n}{z_0-\lambda_n}\right|\le\dfrac{\abs{z}+\abs{\lambda_n}}{\abs{z_0-\lambda_n}}\le\dfrac{2R}{\text{Im}z_0}.$$
    Thus, there exists some $M_1(K)>0$ such that $\abs{\log\abs{\frac{z-\lambda_n}{z_0-\lambda_n}}}\le M_1(K)$. On the other hand, for all those $n\ge1$ such that $\abs{\lambda_n}>R$, we have that
    $$\left|\dfrac{z-\lambda_n}{z_0-\lambda_n}\right|=\left|\dfrac{z-z_0+z_0-\lambda_n}{z_0-\lambda_n}\right|\le1+\dfrac{R+\abs{z_0}}{R-\abs{z_0}}=\dfrac{2R}{R-\abs{z_0}}.$$
    Now, if $R<\abs{\lambda_n}\le2R$, then $\abs{\frac{z-\lambda_n}{z_0-\lambda_n}}\ge\frac{d}{\abs{z_0}+2R}$ whereas if $\abs{\lambda_n}>2R$, we have that $\left|\frac{z-\lambda_n}{z_0-\lambda_n}\right|\ge\frac{\abs{\lambda_n}-R}{\abs{\lambda_n}+R}>\frac{1}{3}$. In total, we may find some $M_2(K)>0$, such that $\abs{\log\abs{\frac{z-\lambda_n}{z_0-\lambda_n}}}\le M_2(K)$. Letting $M>\max\{M_1(K),M_2(K)\}$, we have that 
    \begin{align*}
        \left|\sum_{n=1}^{+\infty}A_n\log\dfrac{z-\lambda_{n}}{z_0-\lambda_n}\right|&\le\sum_{n=1}^{+\infty}\abs{A_n}\abs{\log\left|\dfrac{z-\lambda_{n}}{z_0-\lambda_n}\right|}+\sum_{n=1}^{+\infty}\abs{A_n}\abs{\text{arg}\dfrac{z-\lambda_{n}}{z_0-\lambda_n}}\\
        &\le(M+2\pi)\sum_{n=1}^{+\infty}\abs{A_n}<+\infty
    \end{align*}
    and the desired result follows.
    \end{proof}

The usefulness of the preceding lemma lies in the fact that depending on the choice of the parameters, the series $\sum_{n=1}^{+\infty}A_n\log(z-\lambda_n)$ might not be convergent. We overcome this problem by fixing some $z_0$ and using the  function $\psi$ above. In order to study equation (\ref{PDE1}),  we distinguish the four possible cases of Proposition \ref{possibilities}. According to Proposition \ref{allroots}, all cases can arise and hence we study each case separate in the upcoming subsections. We fix some arbitrary point $z_0\in\mathbb{H}$, according to Lemma \ref{analyticityofh}.

\begin{table}[h]

\begin{tabular}{ | m{5cm} | m{5cm}|  } 
  \hline
  \hspace{15mm}\textbf{Roots of $P_{\mathbb{H}}$} & \hspace{5mm}\textbf{Geometry of the slits}\\
  \hline
  $P_{\mathbb{H}}$ has a complex root $\beta\in\mathbb{H}$. &  $\gamma_n$ spirals about $\beta$, $\forall n\ge1$. \\ 
  \hline
 $P_{\mathbb{H}}$ has a double root $\rho_0\in\mathbb{R}$. & $\gamma_n$ intersects $\mathbb{R}$ tangentially at $\rho_0$, $\forall n\ge1$.\\
  \hline
$P_{\mathbb{H}}$ has a triple root $\rho_0\in\mathbb{R}$. & $\gamma_n$ intersects $\mathbb{R}$ orthogonally at $\rho_0$, $\forall n\ge1$.\\
  \hline
  $P_{\mathbb{H}}$ has distinct real roots and $\exists\rho_1\in\mathbb{R}$,  satisfying $P_{\mathbb{H}}'(\rho_1)>0$. &  $\gamma_n$ intersects $\mathbb{R}$ non-tangentially at some $\rho_1$, $\forall n\ge1$.\\
  \hline
\end{tabular}
\caption{All options for the roots of $P_{\mathbb{H}}$ (left column) and the corresponding behaviour of the slits (right column).}

\end{table}

\subsection{Spirals.} At first, assume that $P_{\mathbb{H}}$ has a complex root $\beta\in\mathbb{H}$. Utilizing Proposition $\ref{possibilities}$, equation $(\ref{differential})$ may be written as 
$$\left(\dfrac{1}{v-\beta}+\dfrac{e^{2i\psi}}{v-\bar{\beta}}-\sum_{n=1}^{+\infty}\dfrac{\abs{\frac{P_{\mathbb{H}}'(\beta)}{P_{\mathbb{H}}'(\lambda_n)}}e^{i\psi}}{v-\lambda_n}\right)dv=\dfrac{P_{\mathbb{H}}'(\beta)}{2(1-t)}dt.$$
Setting $\alpha_n:=\abs{\frac{P_{\mathbb{H}}'(\beta)}{P_{\mathbb{H}}'(\lambda_n)}}$, it is straightforward by Proposition \ref{possibilities} that the sequence $(-a_ne^{i\psi})_{n\ge1}$ is absolutely summable. Hence, through a slight reformulation of Lemma \ref{analyticityofh}, the infinite product $\prod_{n=1}^{+\infty}(\frac{z-\lambda_n}{z_0-\lambda_n})^{-\alpha_ne^{i\psi}}$ is well defined. Therefore, the preceding equation gives the implicit solution $h(v(z,t))=(1-t)^{-\frac{P_{\mathbb{H}}'(\beta)}{2}}h(z)$, where
\begin{equation}\label{koenigs1}
    h(z)=e^{\phi(z)}=\dfrac{z-\beta}{z_0-\beta}\left(\dfrac{z-\bar{\beta}}{z_0-\bar{\beta}}\right)^{e^{2i\psi}}\prod_{n=1}^{+\infty}\left(\dfrac{z-\lambda_n}{z_0-\lambda_n}\right)^{-\alpha_ne^{i\psi}}
\end{equation}
for all $z\in\mathbb{H}$.
\begin{proposition}\label{univalenceofh1}
    The function $h$ defined by (\ref{koenigs1}) is analytic and univalent in the upper half-plane. In particular, $h$ is a $\psi$-spirallike function of $\mathbb{H}$, with $\psi\in(-\frac{\pi}{2},\frac{\pi}{2})$.
\end{proposition}
\begin{proof}
    As before, by Lemma \ref{analyticityofh}, the infinite product converges locally uniformly, therefore analyticity follows directly. In general, if a sequence of analytic functions converges locally uniformly, then so does the sequence of the derivatives. Hence, differentiating $h$, we interchange limit and derivative and then apply Proposition \ref{possibilities} to get 
    $$\dfrac{h'(z)}{h(z)}=\dfrac{1}{z-\beta}+\dfrac{e^{2i\psi}}{z-\bar{\beta}}+\sum_{n=1}^{+\infty}\dfrac{-\alpha_ne^{i\psi}}{z-\lambda_n}=\dfrac{P'_{\mathbb{H}}(\beta)}{P_{\mathbb{H}}(z)}$$
for all $z\in\mathbb{H}$. Using Lemma \ref{FGH} and keeping in mind that $\textrm{Im}(F(z,\beta))<0$ for all $z\in\mathbb{H}$, we have that 
$$\text{Im}\left(e^{-i\psi}\dfrac{(z-\beta)(z-\bar{\beta})h'(z)}{h(z)}\right)=\text{Im}\left(\dfrac{\abs{P_{\mathbb{H}}'(\beta)}}{F(z,\beta)}\right)>0,$$
for all $z\in\mathbb{H}$. As a consequence, according to Theorem \ref{spirallikeinH}, $h$ is a $\psi$-spirallike (and by extension univalent) function of the upper half-plane.
\end{proof}
Using the preceding proposition, the univalence of $h$ shows that 
$$v(z,t)=h^{-1}\left((1-t)^{-\frac{P'_\mathbb{H}(\beta)}{2}}h(z)\right),$$
for all $z\in\mathbb{H}$. In other words, returning back to $w=(1-t)^{\frac{1}{2}}v$ and then to $f=w^{-1}$, we have that the PDE (\ref{PDE1}) is satisfied by the function 
$$f(z,t)=h^{-1}\left((1-t)^{\frac{P'_\mathbb{H}(\beta)}{2}}h\left((1-t)^{-\frac{1}{2}}z\right)\right),$$
for all $z\in\mathbb{H}$ and $t\in[0,1)$.

\begin{rmk}\label{welldefined}
    Note at this point that $f$ is independent of the choice of $z_0\in\mathbb{H}$. Indeed, if we denote by $\phi_{z_0}(z)=\int_{z_0}^{z}\frac{1}{P_{\mathbb{H}}(\zeta)}d\zeta$, then we have that $\phi'_{z_0}=\phi'_{z'_0}$, for any other choice $z_0'$. Hence, $h_{z_0}=ch_{z_0'}$, for some constant number $c$. However, due to the conjugation formula above, we have that 

\begin{align*}
    f_{z_0}(z,t):&=h_{z_0}^{-1}\left((1-t)^{\frac{P'_\mathbb{H}(\beta)}{2}}h_{z_0}\left((1-t)^{-\frac{1}{2}}z\right)\right)\\
    &=h_{z_0'}^{-1}\left((1-t)^{\frac{P'_\mathbb{H}(\beta)}{2}}h_{z_0'}\left((1-t)^{-\frac{1}{2}}z\right)\right)=:f_{z_0'}(z,t)
\end{align*}
and thus $f$ is well defined. 
\end{rmk}

\begin{lemma}\label{image of h1}
    The function $h$ defined by $(\ref{koenigs1})$ maps the upper half-plane onto the complement of infinitely many logarithmic spirals joining the tip points $h(k_n)$ to the point at infinity. Moreover, $h(x)\rightarrow\infty$, as $x\rightarrow\infty$.
\end{lemma}
\begin{proof}
    By Proposition \ref{univalenceofh1}, it suffices to find the image of the real line under $h$. By $(\ref{koenigs1})$, using the elementary trigonometric identities $e^{2i\psi}+1=2\cos(\psi)e^{i\psi}$ and $e^{2i\psi}-1=2i\sin(\psi)e^{i\psi}$, a series of straightforward computations shows that for each $m\ge1$, $h(x)=h(k_m)C\exp{(e^{i\psi}S_m(x)-ie^{i\psi}\sum_{n=1}^{+\infty}a_n\mathrm{Arg}(\frac{x-\lambda_n}{k_m-\lambda_n}))}$, where
    $$S_m(x)=2\log\left|\dfrac{x-\beta}{k_m-\beta}\right|\cos(\psi)+2\mathrm{Arg}\left(\dfrac{x-\beta}{k_m-\beta}\right)\sin(\psi)-\sum_{n=1}^{+\infty}a_n\log\left|\dfrac{x-\lambda_n}{k_m-\lambda_n}\right|$$
    and $C\in\mathbb{C}$ is some absolute constant. Assume, for the sake of simplicity, that $C=1$ (see Remark \ref{welldefined}). Now, let $\lambda_{m^*}\in\mathbb{R}$ be the largest $\lambda_j$, so that $k_m\in(\lambda_{m^*},\lambda_m)$. Restricting this interval, we see that $h(x)=h(k_m)\exp(e^{i\psi}S_m(x))$, which implies the first part of the lemma. 
    
    For the second part, again by straightforward calculations, we have that
    \begin{align*}
        \log\abs{h(x)}&=\cos(\psi)\left(2\log\abs{x-\beta}\cos(\psi)-\sum_{n=1}^{+\infty}a_n\log\left|\dfrac{x-\lambda_n}{z_0-\lambda_n}\right|\right)\\
        &+2\cos(\psi)\sin(\psi)\mathrm{Arg}(x-\beta)+\sin(\psi)\sum_{n=1}^{+\infty}a_n\mathrm{Arg}\left(\dfrac{x-\lambda_n}{z_0-\lambda_n}\right).
    \end{align*}
Recall that the only possible accumulation points of the sequence $(\lambda_n)_{n\ge1}$ are $\pm\infty$. Let us assume, without loss of generality, that it accumulates at both $+\infty$ and $-\infty$. We will show that $\abs{h(x)}\rightarrow\infty$, as $x\rightarrow+\infty$. Consider large enough $x>0$, say $x>\mathrm{Re}\beta+1$. By the first part of Proposition \ref{possibilities} and the fact the $\cos(\psi)>0$, we get that 

 \begin{align*}
        \log\abs{h(x)}&\approx\log\abs{x-\beta}\abs{P_{\mathbb{H}}'(\beta)}+\sum_{n=1}^{+\infty}a_n\log\left|\dfrac{x-\beta}{x-\lambda_n}(z_0-\lambda_n)\right|\\
        &\ge\log\abs{x-\beta}\abs{P_{\mathbb{H}}'(\beta)}+\sum_{n=1}^{+\infty}a_n\log\left|\dfrac{x-\mathrm{Re}\beta}{x-\lambda_n}(z_0-\lambda_n)\right|,
    \end{align*}
where the symbol $\approx$ denotes comparability. We will prove that the preceding sum is bounded below, therefore the right-hand side tends to infinity. To see this, we split the sum to those indices such that $\lambda_n>\mathrm{Re}\beta+1$ and the rest. It is easy to see that for those indices, we have that $\abs{\frac{x-\mathrm{Re}\beta}{x-\lambda_n}}\ge\min\{1,\frac{1}{\lambda_n-\mathrm{Re}\beta-1}\}=\frac{1}{\lambda_n-\mathrm{Re}\beta-1}$, for large $n$, since $(\lambda_n)_{n\ge1}$ accumulate at $\infty$. This implies that
$$\sum_{\lambda_n\ge\mathrm{Re}\beta+1}a_n\log\left|\dfrac{x-\mathrm{Re}\beta}{x-\lambda_n}(z_0-\lambda_n)\right|\ge\sum_{\lambda_n\ge\mathrm{Re}\beta+1}a_n\log\left|\dfrac{z_0-\lambda_n}{\mathrm{Re}\beta+1-\lambda_n}\right|.$$
The case for the indices so that $\lambda_n<0$ is easier and one can see that $\log|\frac{x-\mathrm{Re}\beta}{x-\lambda_n}|$ is positive. Combining everything together, we deduce that $\abs{h(x)}\rightarrow\infty$, as $x\rightarrow+\infty$ and similarly when $x$ tends to $-\infty$, the result follows. 
\end{proof}

 \begin{figure}[h]
        \centering
        \includegraphics[scale=0.4]{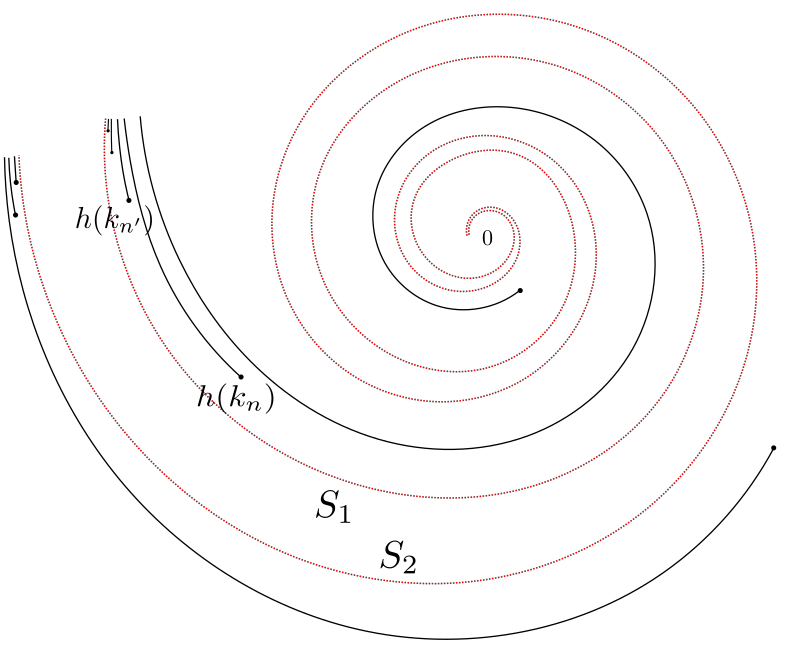}
        \caption{The image of $h$, when $P_{\mathbb{H}}$ has a complex root.}
        \label{fig:spirallikesector}
    \end{figure}

By Lemma \ref{image of h1} we see that the image of $h$ is the complement of infinitely many spirals with tip points $h(k_n)$, so that $h(h_n)\rightarrow\infty$, as in Figure \ref{fig:spirallikesector}. One can also show that there exists a \textit{spirallike sector} of angle $\psi$ and amplitude $\frac{1}{\cos(\psi)}\sum_{n=1}^{+\infty}a_n\pi$, so that it does not contain any spirals $\Gamma(s)=h(k_n)e^{e^{i\psi}s}$, $s\in\mathbb{R}$ (see Figure \ref{fig:spirallikesector}). Such a sector is defined as the set $\mathrm{Spir}[\psi,\alpha,\theta_0]=\{e^{e^{i\psi}t+i\theta}:t\in\mathbb{R}, \theta\in(\theta_0-\alpha,\theta_0+\alpha)\}$ (see \cite{bra}, p. 385). Note that in order to determine the amplitude of the sector, this is done by computing the amplitude of the points of intersection of the boundary spirals $S_1$, $S_2$ of the sector, with the unit circle. Thus, if $\Gamma(s_{0,n})\in\partial\mathbb{D}$, where $s_{0,n}=-\frac{1}{\cos(\psi)}\log\abs{h(k_n)}$, we then let $\zeta_1:=\lim_{k_n\rightarrow+\infty}\Gamma(s_{0,n})$ and $\zeta_2:=\lim_{k_n\rightarrow-\infty}\Gamma(s_{0,n})$ to be the points $S_1\cap\partial\mathbb{D}$ and $S_1\cap\partial\mathbb{D}$ respectively. It is, now, a matter of straightforward calculations to determine the points $\zeta_j\in\partial\mathbb{D}$, thus the amplitude of the sector. 
 
\subsection{Non-tangential intersections.} Assume that $P_{\mathbb{H}}$ has two distinct real roots $\rho_1,\rho_2$ in some interval of $P_{\mathbb{H}}$. Relabeling, if necessary, we assume that $\lambda_m<\rho_1<\rho_2 $ if $\rho_1,\rho_2$ lie in a bounded interval $I_m$, $\rho_1<\rho_2$ if they lie in the left unbounded interval, or $\rho_2<\rho_1$ if they lie in the right unbounded interval. This way, we ensure that  $P'_{\mathbb{H}}(\rho_1)>0>P'_{\mathbb{H}}(\rho_2)$. For the rest we will cover the first two cases where $\rho_1<\rho_2$. Then, the Möbius transform $T(z)=\frac{z-\rho_2}{z-\rho_1}$ maps the upper half-plane onto itself. We shall use this fact shortly. For the case where we have a right unbounded interval, thus $\rho_2<\rho_1$, the results of this section follow similarly by considering $-T$. We know that $P_{\mathbb{H}}'(\lambda_n)<0$ for all $n\ge1$. We then work as follows. With the help of Proposition \ref{possibilities}, the ODE (\ref{differential}) is written as 
$$\left(\dfrac{-1}{v-\rho_1}+\dfrac{\frac{P'_{\mathbb{H}}(\rho_1)}{\abs{P'_{\mathbb{H}}(\rho_2)}}}{v-\rho_2}+\sum_{n=1}^{+\infty}\dfrac{\frac{P'_{\mathbb{H}}(\rho_1)}{\abs{P'_{\mathbb{H}}(\lambda_n)}}}{v-\lambda_n}\right)dv=\dfrac{-P'_{\mathbb{H}}(\rho_1)}{2(1-t)}dt.$$
Hence, considering after integration the function
\begin{equation}
    \label{koenigs2}
    h(z)=\dfrac{(z-\rho_2)^b}{z-\rho_1}\prod_{n=1}^{+\infty}\left(\dfrac{z-\lambda_n}{z_0-\lambda_n}\right)^{a_n}
\end{equation}
for all $z\in\mathbb{H}$, where $b:=\frac{P'_{\mathbb{H}}(\rho_1)}{\abs{P'_{\mathbb{H}}(\rho_2)}}>0$ and $a_n:=\frac{P'_{\mathbb{H}}(\rho_1)}{\abs{P'_{\mathbb{H}}(\lambda_n)}}>0$, $n\ge1$, we deduce the implicit solution $h(v(z,t))=(1-t)^{\frac{P'_\mathbb{H}(\rho_1)}{2}}h(z)$.
\begin{proposition}
    \label{univalenceofh2}
   The function $h$ defined by $(\ref{koenigs2})$ is analytic and univalent in the upper half-plane.
\end{proposition}
\begin{proof}
    By Lemma $\ref{analyticityofh}$, $h$ is well defined and analytic in the upper half-plane. To prove that it is univalent, at first we observe that  $\frac{h'(z)}{h(z)}=\frac{-P'_{\mathbb{H}}(\rho_1)}{P_{\mathbb{H}}(z)}$. Consider, now, the Möbius transform $T(z)=\frac{z-\rho_2}{z-\rho_1}:\mathbb{H}\rightarrow\mathbb{H}$. Taking into account the elementary identity  $(T^{-1}(z)-\rho_1)(T^{-1}(z)-\rho_2)=(\rho_2-\rho_1)z(T^{-1})'(z)$, we then deduce that
    $$\dfrac{(h\circ T^{-1})'(z)}{h\circ T^{-1}(z)}=\dfrac{P'_{\mathbb{H}}(\rho_1)}{\rho_1-\rho_2}\dfrac{(T^{-1}(z)-\rho_1)(T^{-1}(z)-\rho_2)}{zP_{\mathbb{H}}(T^{-1}(z))}=\dfrac{P'_{\mathbb{H}}(\rho_1)}{\rho_1-\rho_2}\dfrac{1}{zH(T^{-1}(z),\rho_1,\rho_2)}$$
for all $z\in\mathbb{H}$, where $H$ is given by Lemma $\ref{FGH}$. Next, we shall show that the above quotient has constant sign in the upper half-plane. It suffices to show that $\text{Im}(T(z)H(z,\rho_1,\rho_2))<0$ for all $z\in\mathbb{H}$. But we observe, again by Lemma \ref{FGH}, that $T(z)H(z,\rho_1,\rho_2)=G(z,\rho_1)$. As a result, since $\rho_1$ is a root and since $P'_{\mathbb{H}}(\rho_1)>0$, we have that $\text{Im}G(z,\rho_1)<0$. Thus, the imaginary part of the logarithmic derivative of $h\circ T^{-1}$ has constant sign, so it is univalent and this concludes the proof.
\end{proof}

\begin{lemma}\label{image of h2}
The function $h$ maps the upper half-plane onto $\mathbb{H}\setminus\bigcup_{n=1}^{\infty}[0,h(k_n)]$, where $(0,h(k_n)]\subset\mathbb{H}$ is a line segment emanating from the origin to the tip point $h(k_n)$. Moreover, we have that $h(k_n)\rightarrow0$, as $n\rightarrow+\infty$.
\end{lemma}

\begin{proof}
     First of all, we observe that the sequences $(\lambda_n)_{n\ge1}$ and $(k_n)_{n\ge1}$ do not have any accumulation points in $\mathbb{R}$. So, there exists some $\epsilon>0$ sufficiently small such that $\lambda_n,k_n\notin[\rho_1-\epsilon,\rho_1+\epsilon]$, for all $n\in\mathbb{N}$. On the contrary, both $(\lambda_n)_{n\ge1}$ and $(k_n)_{n\ge1}$ accumulate at $-\infty$ and/or $+\infty$. Therefore, we are interested in the limit $\lim\limits_{\mathbb{R}\ni x\to\infty}h(x)$. To calculate it, we consider the function $g(x)=h(T^{-1}(x))=h(\frac{x\rho_1-\rho_2}{x-1})$, $x\in\mathbb{R}\setminus\{1\}$. Obviously, $\lim\limits_{x\to1}g(x)=\lim\limits_{x\to\pm\infty}h(x)$. After several simple computations, we find that
    $$g(x)=x^b\frac{(x-1)^{1-b}}{(\rho_1-\rho_2)^{1-b}}\prod\limits_{n=1}^{+\infty}\left(\frac{\rho_1-\lambda_n}{z_0-\lambda_n}\frac{x-T(\lambda_n) }{x-1}\right)^{a_n}.$$
    Observe that the sequence $(T(\lambda_n))_{n\ge1}$ is a bounded sequence of real points and beacuse the series $\sum_{n=1}^{+\infty}a_n\log\abs{\frac{\rho_1-\lambda_n}{z_o-\lambda_n}}$ converges by Lemma \ref{analyticityofh}, taking the absolute value, we write
    \begin{align*}
    \log\abs{g(x)}&=(b-1)\log\abs{\rho_1-\rho_2}+\sum_{n=1}^{+\infty}a_n\log\abs{\frac{\rho_1-\lambda_n}{z_o-\lambda_n}}+b\log\abs{x}\\
    &+(1-b-\sum_{n=1}^{+\infty}a_n)\log\abs{x-1}+\sum_{n=1}^{+\infty}a_n\log\abs{x-T(\lambda_n)}.
    \end{align*}
    
    Now, since $T(\lambda_n)\rightarrow1$, there exists some $N>1$, such that $\abs{T(\lambda_n)-1}<\frac{1}{2}$, for all $n\ge N$. So, for all $x\in(\frac{1}{2},\frac{3}{2})$, we have that $\log\abs{x-T(\lambda_n)}$ is negative for all $n\ge N$. Note that $1-b-\sum_{n=1}^{+\infty}a_n=P_{\mathbb{H}}'(\rho_1)>0$, due to Proposition \ref{possibilities}. Gathering everything together, we obtain
    
     \begin{align*}
    \log\abs{g(x)}&=C+b\log\abs{x}+\sum_{n=1}^{N-1}a_n\log\abs{x-T(\lambda_n)}\\
    &+P_{\mathbb{H}}'(\rho_1)\log\abs{x-1}+\sum_{n=N}^{+\infty}a_n\log\abs{x-T(\lambda_n)}\rightarrow-\infty,
    \end{align*}
    as $x\to1$. As a result, $\lim\limits_{x\to1}g(x)=0$, which leads to $\lim\limits_{x\to+\infty}h(x)=\lim\limits_{x\to-\infty}h(x)=0$. This implies that $\lim_{n\to+\infty}h(k_n)=0$, as well, and we have the desired result. 
    
    On another note, let $M>0$. Then, there exists $\delta>0$ such that $|h(x)|<\delta$, for all $x\in(-\infty,-M]\cup[M,+\infty)$. In addition, due to compactness and the fact that $\rho_1$ is the only singularity of $h$ in $\mathbb{R}$, there also exists some $\delta'$ such that $|h(x)|<\delta'$, for all $x\in[-M,M]\setminus[\rho_1-\epsilon,\rho_1+\epsilon]$. Picking $R=\max\{\delta,\delta'\}$, we have that $|z|<R$, for all $z\in\bigcup_{n=1}^{+\infty}[0,h(k_n)]$.

    Now, regarding the image of $h$, analyzing $g(x)$, we find that 
    $$\mathrm{Arg}(g(x))=C'+b\mathrm{Arg}(x)+P_{\mathbb{H}}'(\rho_1)\mathrm{Arg}(x-1)+\sum_{n=1}^{+\infty}a_n\mathrm{Arg}(x-T(\lambda_n))$$
    for some constant $C'$. As we saw in Remark \ref{welldefined}, we may assume that $C'=0$. Because $T(x)=\frac{x-\rho_2}{x-\rho_1}$ and because there exists no $\lambda_n\in(\rho_1,\rho_2)$, we have that $T(\lambda_n)>0$, for all $n\ge1$ ($T$ maps the line segment $(\rho_1,\rho_2)$ onto the negative line). Therefore, we can see that for $x<0$, 
    $$\mathrm{Arg}(g(x))=b\pi+P'_\mathbb{H}(\rho_1)\pi+\sum_{n=1}^{+\infty}a_n\pi=\pi P'_\mathbb{H}(\rho_1)\left(1-\frac{1}{P'_\mathbb{H}(\rho_2)}-\frac{1}{P'_\mathbb{H}(\lambda_n)}\right)=\pi,$$
    through Proposition \ref{possibilities}. Thus, $g$ maps the negative line onto itself. To see the image of the positive line, we consider the angles $$\theta_1:=P_{\mathbb{H}}'(\rho_1)\pi+\sum_{n:T(\lambda_n)>1}a_n\pi \quad \text{and} \quad\theta_2:=\sum_{n:T(\lambda_n)>1}a_n\pi.$$

    \begin{figure}
        \centering
        \includegraphics[scale=0.4]{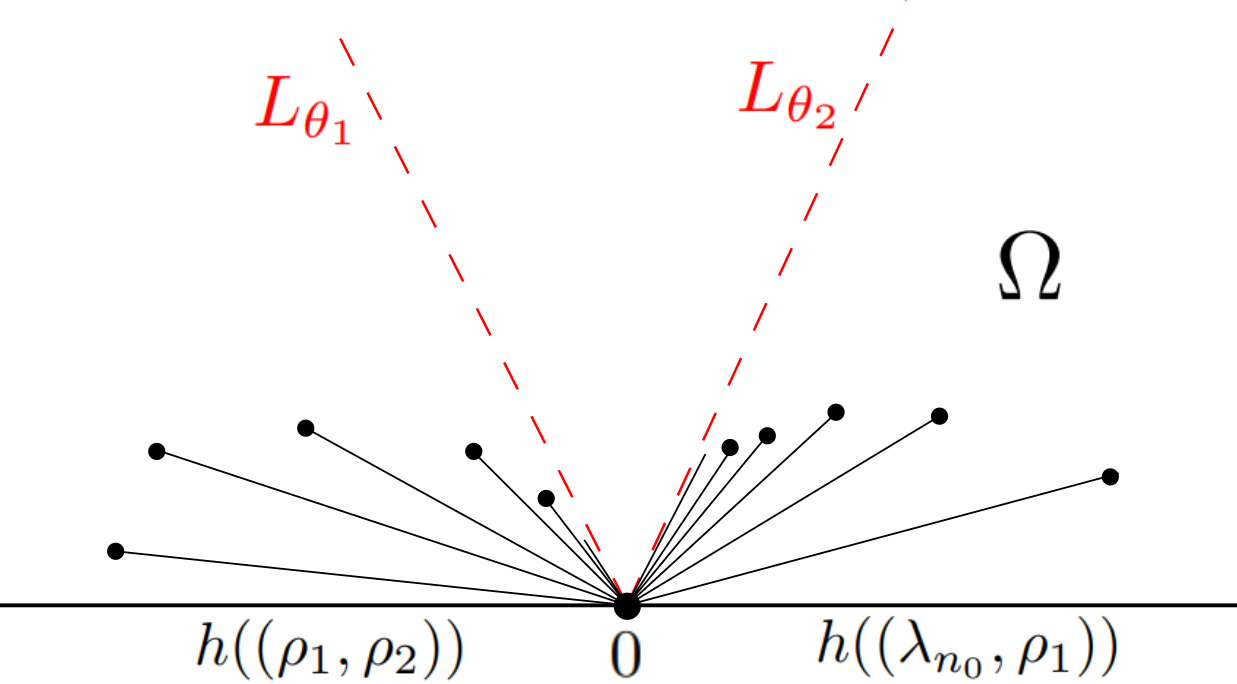}
        \caption{The image of $h$, when $P_{\mathbb{H}}$ has distinct real roots. There exists an angular sector of amplitude $P_{\mathbb{H}}'(\rho_1)\pi$, that contains no slits.}
        \label{fig:sector}
    \end{figure}
    
    We then deduce, similarly as before, that $g$ maps $[0,1]$ onto the union of line segments of the form $[0,h(k_n)]$, for all $n\ge1$, so that $T(\lambda_n)<1$. Assume without loss of generality that $T(\lambda_n)<1$ for infinitely many indices. Otherwise, we only have finitely many such segments, so we comment no further. Those segments, accumulate to the half-line $L_{\theta_1}:=\{xe^{i\theta_1}:x\ge0\}$ and by the first part of the proof, the tip points of the segments accumulate at the origin. Similarly, the image of the line $[1,+\infty]$ is the union of the line segments of the form $[0,h(k_n)]$, for all $n\ge1$, so that $T(\lambda_n)>1$, and the half-line $[0,+\infty]$; recall that $(T(\lambda_n))_{n\ge1}$ is bounded, so picking $n_0\ge1$ such that $T(\lambda_{n_0})=\max T(\lambda_n)$, then $g([T(\lambda_{n_0}),+\infty])=[0,+\infty]$. Again, those segments accumulate to the half-line $L_{\theta_2}:=\{xe^{i\theta_2}:x\ge0\}$ with the tip points accumulating at zero (see Figure \ref{fig:sector}). 
    \end{proof}

The Loewner flow in this case is given by the conjugation formula
\begin{equation}
    f(z,t)=h^{-1}\left(\left(1-t\right)^{-\frac{P'_\mathbb{H}(\rho_1)}{2}}h\left(\left(1-t\right)^{-\frac{1}{2}}z\right)\right)
\end{equation}
for all $z\in\mathbb{H}$ and $0\le t<1$. In order to study the geometry of the slits as $t\rightarrow1^-$, we follow the formula above and thus, we extend the tip points $h(k_n)$ to infinity by applying the mapping $z\mapsto e^{\frac{P'_{\mathbb{H}}(\rho_1)}{2(1-t)}}z$ and then we consider their preimage of $\{xh(k_n):x\ge1\}$ under $h$. The following proposition shows that the trajectory of $f(k_n\sqrt{1-t},t)$ collides at the point $\rho_1\in\mathbb{R}$ as $t\rightarrow1^-$, non-tangentially.

\begin{proposition}\label{non-tangential}
    For each $n\in\mathbb{N}$, the trace $\hat{\gamma}_n:=\{f(k_n\sqrt{1-t},t):[0,1)\}$ is a smooth curve intersecting the real line non-tangentially at the root $\rho_1$.
\end{proposition}
\begin{proof}
    Fix $n\in\mathbb{N}$ and assume without loss of generality that $k_n>\rho_1$. Consider the curve $\gamma_n:[0,1)\to\mathbb{H}$ with $\gamma_n(t)=f(k_n\sqrt{1-t},t)$. Surely $\lim\limits_{t\to1}\gamma_n(t)=\rho_1$. We only need to deal with the angle of the convergence and we are going to do this through the use of harmonic measure. By (\ref{koenigs2}), $h$ maps the upper half-plane $\mathbb{H}$ conformally onto the simply connected domain $\Omega:=\mathbb{H}\setminus\bigcup\limits_{j=1}^{+\infty}[0,h(k_j)]$. In addition, the image of $\gamma_n$ through $h$ is the ray $\{re^{i\arg h(k_n)}:r>|h(k_n)|\}$ and $\lim\limits_{t\to+\infty}h(\gamma_n(t))=\infty$. It is easy to see that $h\circ\gamma_n$ separates the prime ends of $\Omega$ into two connected components. The one consists of the prime ends corresponding to $(-\infty,0]$, to $\bigcup\limits_{j\in\mathbb{N}}\{[0,h(k_j)]:\arg h(k_j)>\arg h(k_n)\}$ and the prime ends corresponding to $[0,h(k_n)]$ defined by crosscuts with arguments larger than $\arg h(k_n)$. The other is the complement. We denote them by $\partial\Omega^+$ and $\partial\Omega^-$, respectively (see Figure \ref{fig:non-tang}). 
    
    Our first objective is to prove that $\lim\limits_{t\to1}\omega(h(\gamma_n(t)),\bigcup\limits_{j=1}^{+\infty}[0,h(k_j)],\Omega)=0$. Indeed, by the previous lemma, there exists some $R>0$ such that $|z|<R$, for all $z\in\bigcup\limits_{j=1}^{+\infty}[0,h(k_j)]$. Therefore, through a new conformal mapping $g$ of $\Omega$ onto $\mathbb{H}$ which fixes $\infty$ and sends $h(k_n)$ to $0$, we may map a subset of $\bigcup\limits_{j=1}^{+\infty}[0,h(k_j)]$ onto the closed segment $[a,b]\subset\mathbb{R}$. Then, by conformal invariance,
    $$\lim\limits_{t\to1}\omega(h(\gamma_n(t)),\bigcup\limits_{j=1}^{+\infty}[0,h(k_j)],\Omega)\le\lim\limits_{t\to1}\omega(g(h(\gamma_n(t))),[a,b],\mathbb{H})=0,$$
    in view of Subsection 2.1. Combining this with the domain monotonicity property, we have
    $$\lim\limits_{t\to1}\omega(h(\gamma_n(t)),\partial\Omega^+,\Omega)=\lim\limits_{t\to1}\omega(h(\gamma_n(t)),(-\infty,0],\Omega)\le\lim\limits_{t\to1}\omega(h(\gamma_n(t)),(-\infty,0],\mathbb{H}).$$
    In a similar fashion, 
    $$\lim\limits_{t\to1}\omega(h(\gamma_n(t)),\partial\Omega^-,\Omega)\le\lim\limits_{t\to1}\omega(h(\gamma_n(t)),[0,+\infty),\mathbb{H}).$$

    \begin{figure}
        \centering
        \includegraphics[scale=0.65]{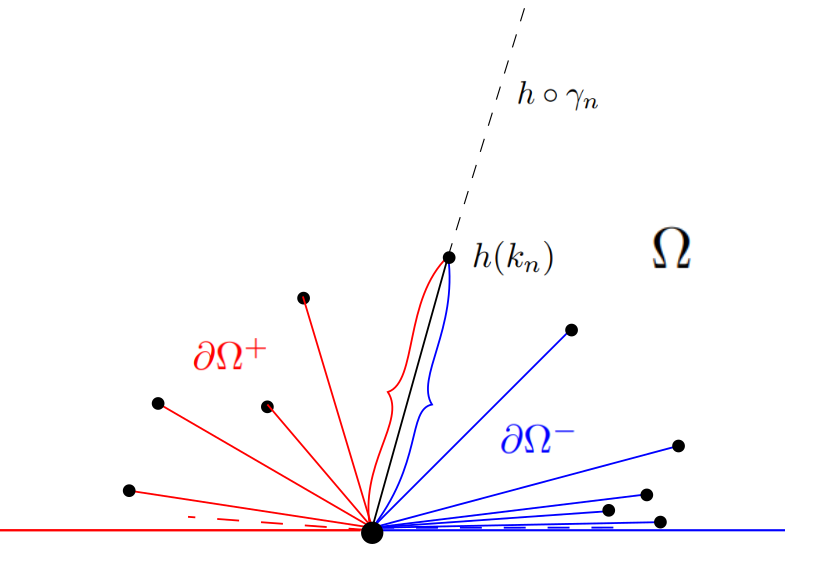}
        \caption{The domain $\Omega$ and its prime ends}
        \label{fig:non-tang}
    \end{figure}

    However, in the last two relations, both the left-hand and right-hand sides add up to 1. This implies that the equalities prevail. Consequently,
    \begin{eqnarray*}
    \lim\limits_{t\to1}\omega(h(\gamma_n(t)),\partial\Omega^-,\Omega)&=&\lim\limits_{t\to1}\omega(h(\gamma_n(t)),[0,+\infty),\mathbb{H})\\
    &=&\lim\limits_{r\to+\infty}\omega(re^{i\arg h(k_n)},\{w\in\mathbb{C}:\arg w=0\},U_{0,\pi})\\
    &=&\lim\limits_{r\to+\infty}\frac{\pi-\arg h(k_n)}{\pi}.
    \end{eqnarray*}
    Using a Riemann mapping of $\Omega$ onto the unit disk, the conformal invariance of the harmonic measure and the preservation of the orientation, Remark \ref{cara} reveals that the image of $h\circ\gamma_n$ inside the unit disk converges to some point of the unit circle by angle $\pi(1-\frac{\pi-\arg h(k_n)}{\pi})=\arg h(k_n)$. Finally, through a M\"{o}bius transformation which preserves angles in the whole complex plane, we return to our initial setting, to see that $\gamma_n$ intersects the real line at $\rho_1$ by angle $\pi-\arg h(k_n)\in(0,\pi)$, thus non-tangentially. Remember that $h(k_n)\in\mathbb{H}$ or equivalently $0<\arg h(k_n)<\pi$. The difference between the angles in the unit disk and the upper half-plane derives from the fact that in the unit disk we use to count the angle with the help of the tangent, while in the upper half-plane the angle is usually counted by starting from the ``positive'' semi-axis.
\end{proof}

\subsection{Tangential intersections.}
Assume, now, that $P_{\mathbb{H}}$ has a double root $\rho_0\in\mathbb{R}$. Fix some $z_0\in\mathbb{H}$ and consider the function $h(z):=\int_{z_0}^{z}\frac{d\zeta}{P_{\mathbb{H}}(\zeta)}$, which is well defined and analytic, since $\mathbb{H}$ is simply connected. Therefore, by $(\ref{double root})$ and due to the uniform convergence on compacta, we have that 
\begin{equation}
    \label{koenigs3}
    h(z)=\sum_{n=1}^{+\infty}A_n\log\dfrac{z-\lambda_{n}}{z_0-\lambda_n}+\left(1-\sum_{n=1}^{+\infty}A_n\right)\log(z-\rho_0)+\dfrac{B}{z-\rho_0}+\text{const.}
\end{equation}
for all $z\in\mathbb{H}$ where the parameters $A_n<0$ and $B>0$ are given by $(\ref{double root})$ and the implicit solution $h(v(z,t))=-\frac{1}{2}\log(1-t)+h(z)$.
 The following proposition shows that $h$ is a well defined analytic and univalent function of the upper half-plane.
\begin{proposition}
\label{univalenceofh3}
    Let $(A_n)_{n\ge1}$ be a summable sequence of negative numbers, $(\lambda_{n})_{n\ge1}$ be a sequence of real points and let also $\rho_0\in\mathbb{R}$, $B\in\mathbb{R}$ and $C>0$. Fix some $z_0\in\mathbb{H}$. Then, the function
    $$\phi(z)=\sum_{n=1}^{+\infty}A_n\log\dfrac{z-\lambda_{n}}{z_0-\lambda_n}+(1-\sum_{n=1}^{+\infty}A_n)\log(z-\rho_0)+\dfrac{B}{z-\rho_0}+\dfrac{C}{(z-\rho_0)^2}$$
    is analytic and univalent in $\mathbb{H}$. 
\end{proposition}

\begin{proof}
Analyticity follows directly from Lemma \ref{analyticityofh}. To prove that $h$ is univalent, we precompose with the Möbius transform $T(z)=\frac{1}{\rho_0-z}$. Note that $T:\mathbb{H}\rightarrow\mathbb{H}$. By differentiating, we get that
    $$(\phi\circ T^{-1} )'(z)=\sum_{n=1}^{+\infty}\dfrac{A_n}{z-T(\lambda_n)}-\dfrac{1}{z}-B+2Cz$$
    which has positive imaginary part for all $z\in\mathbb{H}$. Hence, $\phi\circ T^{-1}$ and by extension $\phi$ are univalent in $\mathbb{H}$, since the upper half-plane is a convex domain.
    \end{proof}

\begin{lemma} \label{boundnessofh3}
Assume that $h$ is given by $(\ref{koenigs3})$. Then, it maps the upper half-plane onto a horizontal half-plane, minus infinitely many horizontal half-lines that extend to infinity from the left. In addition, $\mathrm{Re}h(x)\rightarrow+\infty$, as $x\rightarrow\pm\infty$.
\end{lemma}
\begin{proof}
We argue as in Lemma \ref{image of h2}. Again, the only accumulation points of the sequences  $(\lambda_n)_{n\ge1}$ and $(k_n)_{n\ge1}$ are $\pm\infty$.We, now, consider the function $g(x)=h(T^{-1}(x))=h(\frac{\rho_0x-1}{x})$, $x\in\mathbb{R}\setminus\{0\}$. It is easy to see that $\lim\limits_{x\to0}g(x)=\lim\limits_{x\to\pm\infty}h(x)$. After a series of calculations, we see that
    $$g(x)=\sum\limits_{n=1}^{+\infty}A_n\log\left(\frac{\rho_0-\lambda_n}{z_0-\lambda_n}\frac{x-T(\lambda_n)}{x}\right)+\left(1-\sum\limits_{n=1}^{+\infty}A_n\right)\log\left(-\frac{1}{x}\right)-Bx+\text{const.},$$
    where $B>0$ and $A_n<0$, for every $n\in\mathbb{N}$. Focusing on the real parts, we are led to
    \begin{align*}
    \text{Re}g(x)&=\sum\limits_{n=1}^{+\infty}A_n\log\left|\frac{\rho_0-\lambda_n}{z_0-\lambda_n}\frac{x-T(\lambda_n)}{x}\right|-\left(1-\sum\limits_{n=1}^{+\infty}A_n\right)\log|x|-Bx+\text{const.}\\
    &=\sum\limits_{n=1}^{+\infty}A_n\log\left|\frac{\rho_0-\lambda_n}{z_0-\lambda_n}(x-T(\lambda_n))\right|-\log|x|-Bx+\text{const.}
    \end{align*}

    \begin{figure}
        \centering
        \includegraphics[scale=0.45]{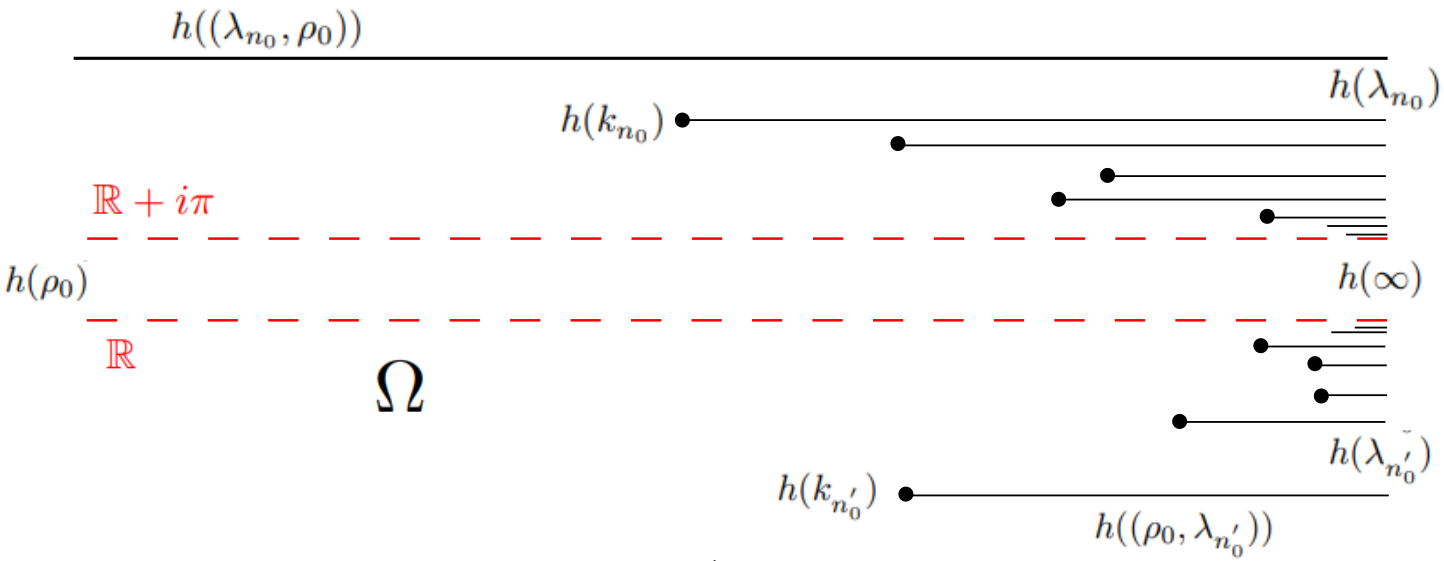}
        \caption{The image of $h$, when $k_{n_0}<\lambda_{n_0}<\rho_0<k_{n_0'}$, for some $n_0\ge1$.}
        \label{fig:positivestep}
    \end{figure}
    
    We focus on the infinite sum. Since the only possible limit points of $(\lambda_n)_{n\ge1}$ are $\pm\infty$ and since $(T(\lambda_n))_{n\ge1}$ is bounded and converges to $0$ as $n\to+\infty$, we have
    $$\sum\limits_{n=1}^{+\infty}A_n\log\left|\frac{\rho_0-\lambda_n}{z_0-\lambda_n}(x-T(\lambda_n))\right|\approx\sum\limits_{n=1}^{N}A_n\log\left|\frac{\rho_0-\lambda_n}{z_0-\lambda_n}(x-T(\lambda_n))\right|+\sum\limits_{n=N+1}^{+\infty}A_n\log|x|,$$
    for some large enough $N\in\mathbb{N}$, where $\approx$ denotes comparability for large $n$. As a result,
    $$\mathrm{Re}g(x)\approx\sum\limits_{n=1}^{N}A_n\log\left|\frac{\rho_0-\lambda_n}{z_0-\lambda_n}(x-T(\lambda_n))\right|-\left(1-\sum\limits_{n=N+1}^{+\infty}A_n\right)\log|x|-Bx+\text{const.}$$
    Keeping in mind that the sequence $(A_n)_{n\ge1}$ is summable with negative terms, we quickly see that $\lim\limits_{x\to0}\mathrm{Re}g(x)=+\infty$. Also, taking into account that $\mathrm{Re}h(\lambda_n)=+\infty$, we deduce that $\lim\limits_{x\to\pm\infty}\mathrm{Re}h(x)=+\infty$.

    Finally, to see the image of $h$, we find $h(\mathbb{R})$. Without loss of generality, we assume that $(k_n)_{n\ge1}$ accumulates at both $\pm\infty$ (and hence so does the sequence $(\lambda_n)_{n\ge1}$) and let also $k_{n_0}<\lambda_{n_0}<\rho_0<k_{n_o'}$. Now, it is direct to see that
    \begin{equation}\label{image of h3}
        \mathrm{Im}h(x)=\sum_{n=1}^{+\infty}A_n\mathrm{Arg}(x-\lambda_n)+(1-\sum_{n=1}^{+\infty}A_n)\mathrm{Arg}(x-\rho_0)+\text{const}.
    \end{equation}
    Again, for simplicity we take the constant to be zero. For any $m\ge1$, so that $\lambda_m<\lambda_{n_0}$, it is $\mathrm{Im}h(x)=(1-\sum_{\lambda_n\le\lambda_m}A_n)\pi$, for all $x\in(\lambda_m,\lambda_{m'})$. Similarly for $\lambda_{m}>\rho_0$, we get that for all $x\in(\lambda_m,\lambda_{m'})$, $\mathrm{Im}h(x)=\sum_{\lambda_n>\lambda_m}A_n\pi$. To conclude, we have that $\mathrm{Im}h(x)=(1-\sum_{\lambda_n\le\rho_0}A_n)\pi$, in $(\lambda_{n_0},\rho_0)$ and because $\mathrm{Re}h(x)\rightarrow-\infty$, as $x\rightarrow\rho_0-$, we deduce that $h$ maps the interval $(\lambda_{n_0},\rho_0)$ onto the line $\mathbb{R}+i\sum_{\lambda_n\le\rho_0}A_n\pi$, as in Figure \ref{fig:positivestep}.
\end{proof}
By the proof of the lemma above, we see that as $\lambda_m\rightarrow -\infty$, then $\mathrm{Im}h(k_m)=(1-\sum_{\lambda_n\le\lambda_m}A_n)\pi\rightarrow\pi$ and as $\lambda_m\rightarrow+\infty$, then $\mathrm{Im}h(k_m)=\sum_{\lambda_n>\lambda_m}A_n\pi\rightarrow0$. Also, $\mathrm{Re}h(k_m)\rightarrow+\infty$ and hence, we see that there is a horizontal strip of amplitude $\pi$ that contains no slits, whereas the tip points of the half-lines accumulate at the boundary of this strip from above and below, while also ``disappearing'' to the right as we see in Figure \ref{fig:positivestep}.

Following similar steps as before, the Loewner flow in this case is given by the conjugation formula
$$f(z,t)=h^{-1}\left(\frac{1}{2}\log(1-t)+h\left((1-t)^{-\frac{1}{2}}z\right)\right).$$
Once again we enquire about the convergence of the corresponding trajectories.
\begin{proposition}\label{tangential}
    For each $n\in\mathbb{N}$, the trace $\hat{\gamma}_n:=\{f(k_n\sqrt{1-t},t):t\in[0,1)\}$ is a smooth curve intersecting the real line tangentially at the root $\rho_0$. In particular, either all curves converge to $\rho_0$ by angle $0$ or all of them converge by angle $\pi$.
\end{proposition}

\begin{proof}
    Again, we will utilize the harmonic measure. Towards this goal, fix $n\in\mathbb{N}$ and consider the curve $\gamma_n:[0,1)\to\mathbb{H}$ with $\gamma_n(t)=f(k_n\sqrt{1-t},t)$. As before, $\lim_{t\to1^-}\gamma_n(t)=\rho_1$. This time, $h$ maps the upper half-plane $\mathbb{H}$ onto a horizontal half-plane minus a sequence of horizontal half-lines stretching to $\infty$ in the positive direction. To be more formal, set $L_j=\{z\in\mathbb{C}:\text{Re}z\ge\text{Re}h(k_j), \text{Im}z=\text{Im}h(k_j)\}$. In our first case, we assume that there exists some $a>\text{Im}h(k_j)$, for all $j\in\mathbb{N}$ such that $\Omega:=h(\mathbb{H})=\{z\in\mathbb{C}:\text{Im}z<a\}\setminus\bigcup_{j=1}^{+\infty}L_j$. By the previous lemma, we may see that there exists $R\in\mathbb{R}$ such that $\text{Re}h(k_j)>R$, for all $j\in\mathbb{N}$. Arguing as before, the curve $\gamma_n$ separates the prime ends of $\Omega$ into two connected components $\partial\Omega^+$ and $\partial\Omega^-$, where $\partial\Omega^+$ consists of the prime ends corresponding to the horizontal line $L:=\{z\in\mathbb{C}:\text{Im}z=a\}$, the prime ends corresponding to $\bigcup_{j\in\mathbb{N}}\{L_j:\text{Im}h(k_j)>\text{Im}h(k_n)\}$ and the prime ends corresponding to the half-line $L_n$ defined by crosscuts with imaginary parts larger than $\text{Im}h(k_n)$. Naturally, $\partial\Omega^-$ consists of all the remaining prime ends. 
    Our objective is to prove that $\lim_{t\to1^-}\omega(h(\gamma_n(t)),\partial\Omega^+,\Omega)=1$. An important note towards this direction is the fact that $h\circ\gamma_n([0,1))=\{z\in\mathbb{C}:\text{Re}z<\text{Re}h(k_n),\text{Im}z=\text{Im}h(k_n)\}$, with $\lim_{t\to1^-}\text{Re}h(\gamma_n(t))=-\infty$.

    \begin{figure}
        \centering
        \includegraphics[scale=0.55]{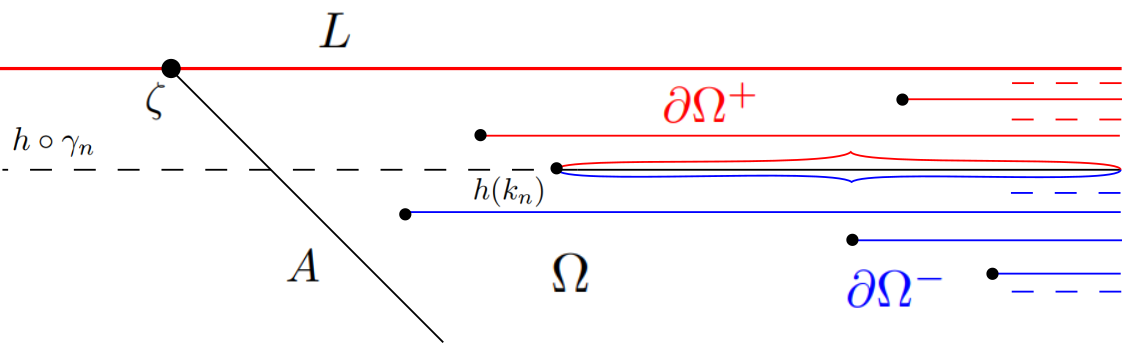}
        \caption{The sets of Proposition \ref{tangential}}
        \label{fig:tang}
    \end{figure}
    
    The existence of the real number $R$ that bounds the real parts of all half-lines, allows us to proceed to the following construction: we may find some point $\zeta\in L$ and a ray $A$ emanating from $\zeta$ such that $A\setminus\{\zeta\}\subset\Omega$ (an obvious example is a point $\zeta$ that rests sufficiently to the left and the half-line $A$ that is perpendicular to $L$ at $\zeta$). Set $L^+=\{z\in L:\text{Re}z\le\text{Re}\zeta\}$ and $L^-=L\setminus L^+$. Then, the angular simply connected domain $U$ bounded by $L^+$ and $A$ is contained inside $\Omega$ (for the whole construction see Figure \ref{fig:tang}). Moreover, it is easy to see that $U$ eventually contains the curve $h\circ\gamma_n$ or equivalently there exists some $t_0\in[0,1)$ such that $h(\gamma_n(t))\in U$, for all $t\in[t_0,1)$. Since $\partial\Omega^+\supset L^+$ and $\Omega\supset U$, we are led to
    $$\omega(h(\gamma_n(t)),\partial\Omega^+,\Omega)\ge\omega(h(\gamma_n(t)),L^+,\Omega)\ge\omega(\gamma_n(t),L^+,U),$$
    for all $t\in[t_0,1)$. By the construction of the ray $A$, there exists some $\beta\in(\pi,2\pi)$ such that
    \begin{eqnarray*}
        \lim\limits_{t\to1}\omega(h(\gamma_n(t)),\partial\Omega^+,\Omega)&\ge&\lim\limits_{t\to1}\omega(h(\gamma_n(t)),L^+,U)\\
        &\ge&\lim\limits_{t\to1}\omega(h(\gamma_n(t))-\zeta, \{w\in\mathbb{C}:\arg w=\pi\},U_{\pi,\beta})\\
        &=&\lim\limits_{t\to1}\frac{\beta-\arg(h(\gamma_n(t))-\zeta)}{\beta-\pi}\\
        &=&1,
    \end{eqnarray*}
    since $\lim_{t\to1}\arg h(\gamma_n(t))-\zeta=\lim_{t\to1}\arg h(\gamma_n(t))=\pi$. However, the harmonic measure has $1$ as an upper bound and as a result
    $$\lim\limits_{t\to1}\omega(h(\gamma_n(t)),\partial\Omega^+,\Omega)=1=\frac{\pi}{\pi}.$$
    Using again a Riemann map of $\Omega$ onto the unit disk, Remark \ref{cara} and a M\"{o}bius transformation to return to $\mathbb{H}$, we see that each $\gamma_n$ coverges to $\rho_1$ by angle $0$ and thus, tangentially.
    Finally, if at the start, our domain was of the form $\Omega:=\{z\in\mathbb{C}:\text{Im}z>a\}\setminus\bigcup_{j=1}^{+\infty}L_j$, then each curve $\gamma_n$ would converge to $\rho_1$ by angle $\pi$, again tangentially.
\end{proof}

\subsection{Orthogonal intersections.} In the final case, we assume that $P_{\mathbb{H}}$ has a triple root $\rho_0\in\mathbb{R}$, which means according to the preliminary analysis, that $\rho_0$ is a double root coinciding with some $\lambda_{n_0}$, hence triple. Integrating in $(\ref{differential})$ and using $(\ref{double root})$, we deduce that $h(v(z,t))=-\frac{1}{2}\log(1-t)+h(z)$, where
\begin{equation}
    \label{koenigs4}
   h(z)=\sum_{n\neq n_0}A_n\log\dfrac{z-\lambda_{n}}{z_0-\lambda_n}+(1-\sum_{n\neq n_0}\log(z-\rho_0)+\dfrac{B}{z-\rho_0}+\dfrac{C}{(z-\rho_0)^2}.
\end{equation}
where the parameters $A_n<0$, $B$ and $C>0$ are given by $(\ref{triple root})$. By Lemma \ref{analyticityofh} and Proposition \ref{univalenceofh3}, $h$ is analytic and univalent in the upper half-plane.

\begin{lemma}\label{boundnessofh4} Assume that $h$ is given by $(\ref{koenigs4})$. Then it maps the upper half-plane onto the complement of inifinitely many horizontal half-lines and we have that:
    \begin{enumerate}
        \item[\textup{(1)}] $\mathrm{Re}h(x)\rightarrow+\infty$, as $x\rightarrow\pm\infty$ and
        \item[\textup{(2)}] there exists some $Q>0$, so that $|\mathrm{Im}h(k_n)|<Q$, for all $n\in\mathbb{N}$.
    \end{enumerate}
\end{lemma}

\begin{figure}
    \centering
    \includegraphics[scale=0.4]{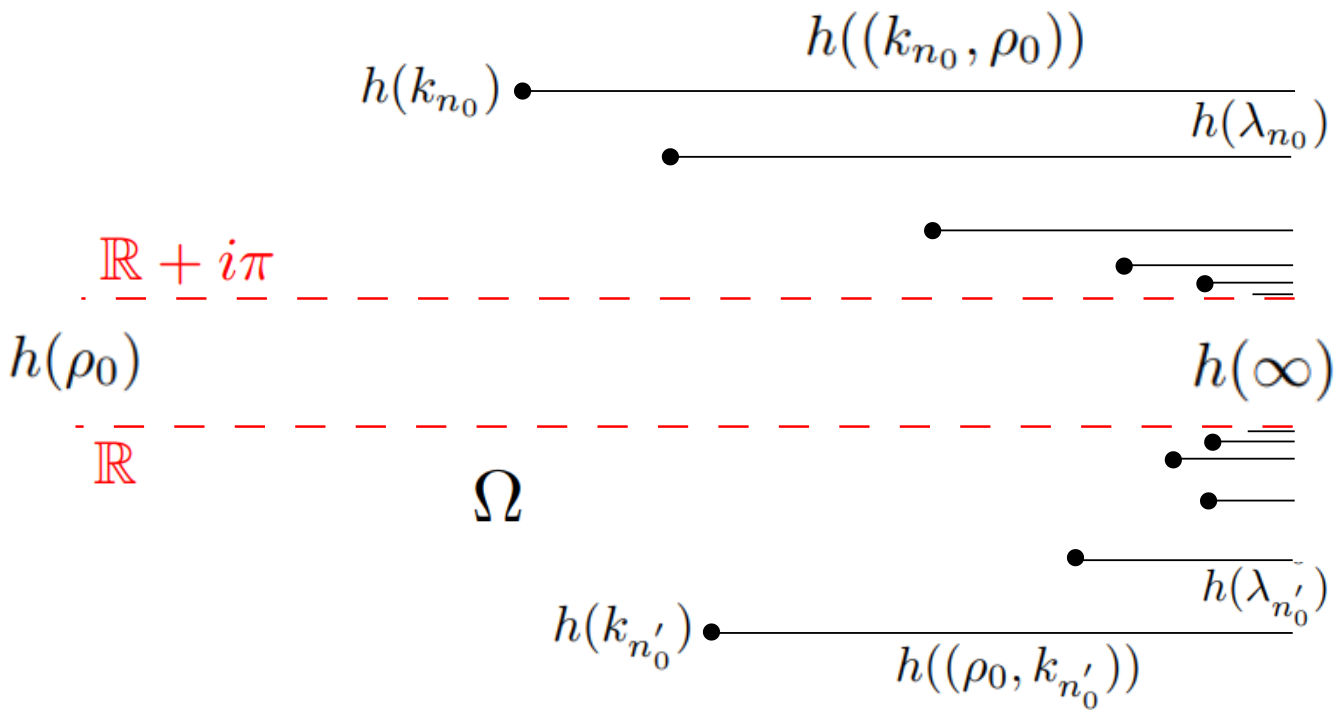}
    \caption{The image of $h$, when $\rho_0$ is a triple root.}
    \label{fig:zerostep}
\end{figure}

\begin{proof}
The proof of the lemma follows an identical procedure to the lemma of the preceding section, albeit with minor modifications. For the sake of brevity, we skip this proof, but we only note some adjustments. We write, without loss of generality, $k_{n_0}<\lambda_{n_0}=\rho_0<k_{n_0'}$. Through the fact that $\mathrm{Re}h(x)$ equals
     $$\sum_{n\neq n_0}A_n\log\left|\frac{x-\lambda_n}{z_0-\lambda_n}\right|+\dfrac{(1-\sum_{n\neq n_0}A_n)\log\abs{x-\rho_0}(x-\rho_0)^2+B(x-\rho_0)+C}{(x-\rho_0)^2}$$
     we see that $\mathrm{Re}h(x)\rightarrow+\infty$, as $x\rightarrow\rho_0$, while by $(\ref{image of h3})$ we see that $h((k_{n_0},\rho_0))$ and $h((\rho_0,k_{n_0'}))$ are the half-lines lying above and below all other half-lines, respectively, as we see in Figure \ref{fig:zerostep}.
    \end{proof}
    Again, we see that the strip $\{0<\mathrm{Im}z<\pi\}$ contains no half-lines, while the tip points accumulate on the boundary of the strip and the point at infinity towards the right. In addition, the Loewner flow is the same as in the previous case. Finally, the following proposition allows us to see that the trajectories of $f(k_n\sqrt{1-t},t)$ collide to the real line, orthogonally at the point $\rho_0$, for all $n\ge1$.

\begin{proposition}\label{orthogonal}
    For each $n\in\mathbb{N}$, the trace $\hat{\gamma}_n:=\{f(k_n\sqrt{1-t},t):t\in[0,1)\}$ is a smooth curve intersecting the real line orthogonally at the root $\rho_0$ (i.e. with angle $\frac{\pi}{2}$).
\end{proposition}
\begin{proof}
    Fix $n\in\mathbb{N}$ and set $L_j=\{z\in\mathbb{C}:\text{Re}z\ge\text{Re}h(k_j),\text{Im}z=\text{Im}h(k_j)\}$. This time, (\ref{koenigs4}) dictates that $\Omega:=h(\mathbb{H})=\mathbb{C}\setminus\bigcup_{j=1}^{+\infty}L_j$. In a similar fashion as in the previous cases, we denote by $\partial\Omega^+$ the prime ends of $\Omega$ corresponding to $\bigcup_{j\in\mathbb{N}}\{L_j:\text{Im}h(k_j)>\text{Im}h(k_n)\}$ along with the prime ends corresponding to the half-line $L_n$ and defined by crosscuts with imaginary parts larger than $\text{Im}h(k_n)$. Once more, we denote by $\partial\Omega^+$ the set of the remaining prime ends. Recall that $h\circ\gamma_n([0,1))=\{z\in\mathbb{C}:\text{Re}z<\text{Re}h(k_n),\text{Im}z=\text{Im}h(k_n)\}$ and $\lim_{t\to1^-}\text{Re}h(\gamma_n(t))=-\infty$. Our aim is to prove that
    $$\lim\limits_{t\to1}\omega(h(\gamma_n(t)),\partial\Omega^+,\Omega)=\lim\limits_{t\to1}\omega(h(\gamma_n(t)),\partial\Omega^-,\Omega)=\frac{1}{2}.$$ 
    By the previous lemma, we see that there exist $n_1,n_2\in\mathbb{N}$ such that $\text{Im}h(k_{n_1})\le\text{Im}h(k_j)\le\text{Im}h(k_{n_2})$, for all $j\in\mathbb{N}$. So, denote by $L^+\subset\partial\Omega^+$ the prime ends corresponding to the half-line $L_{n_2}$ and defined by crosscuts with imaginary parts larger than $\text{Im}h(k_{n_2})$. On the other side. denote by $L^-\subset\partial\Omega^-$ the prime ends corresponding to $L_{n_1}$ and defined by crosscuts with imaginary parts smaller than $\text{Im}h(k_{n_1})$. As in the tangential case, there exists some $R\in\mathbb{R}$ such that $\text{Re}h(k_j)\ge R$, for all $j\in\mathbb{N}$. Therefore, through the conformal invariance of the harmonic measure and a mapping of $\Omega$ onto $\mathbb{H}$ that maps $h(k_n)$ to $0$ and fixes infinity, we may see that
    $$\lim\limits_{t\to1}\omega(h(\gamma_n(t)),\partial\Omega^+\setminus L^+,\Omega)=\lim\limits_{t\to1}\omega(h(\gamma_n(t)),\partial\Omega^-\setminus L^-,\Omega)=0.$$
    
    \begin{figure}
        \centering
        \includegraphics[scale=0.43]{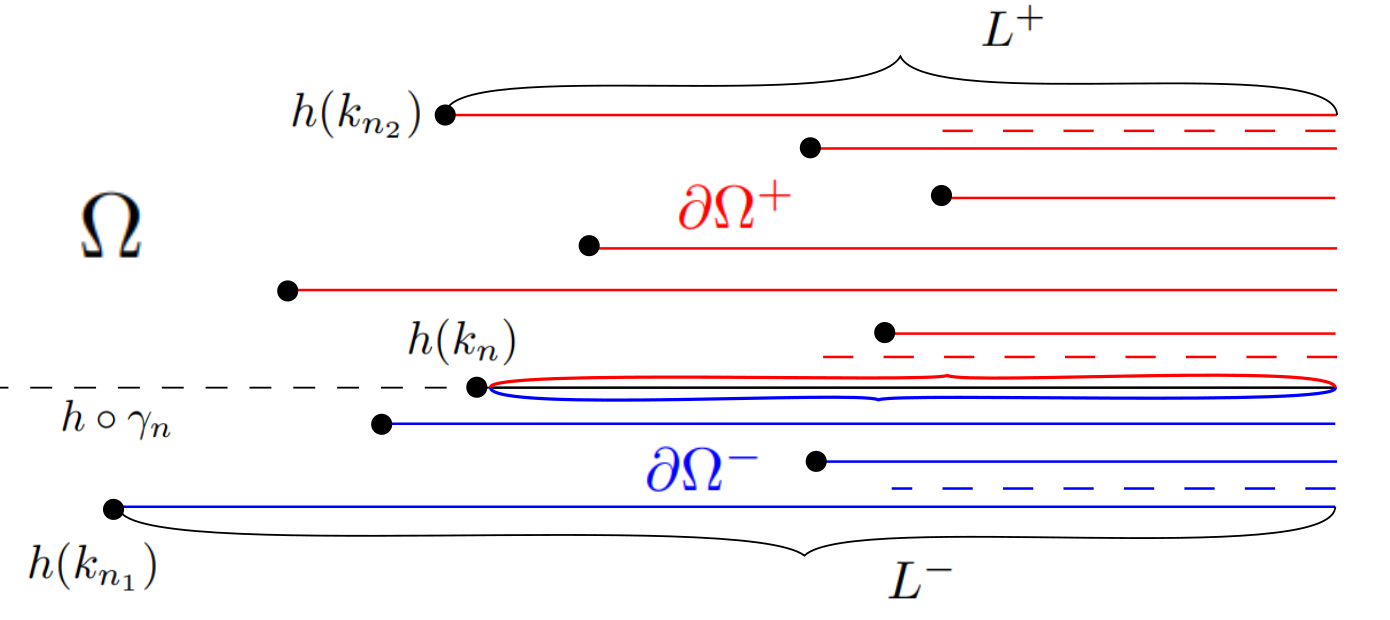}
        \caption{The sets of Proposition \ref{orthogonal}}
        \label{fig:orth}
    \end{figure}

    We will use two auxiliary ``Koebe-like'' domains to reach the desired conclusions. We set $\Omega_1=\mathbb{C}\setminus L_{n_1}$ and $\Omega_2=\mathbb{C}\setminus L_{n_2}$. Trivially, both of them are supersets of $\Omega$. In particular, in the sense of prime ends, $L^-\subset\partial\Omega\cap\partial\Omega_1$ and $L^+\subset\partial\Omega\cap\partial\Omega_2$ (see Figure \ref{fig:orth}). By the domain monotonicity property of harmonic measure,
    $$\lim\limits_{t\to1}\omega(h(\gamma_n(t)),L^-,\Omega_1)\ge\lim\limits_{t\to1}\omega(h(\gamma_n(t)),L^-,\Omega)=\lim\limits_{t\to1}\omega(h(\gamma_n(t)),\partial\Omega^-,\Omega).$$
    Through a chain of conformal mappings, we are going to estimate the harmonic measure with regard to the domain $\Omega_1$. As a matter of fact,
    \begin{eqnarray*}
        \lim\limits_{t\to1}\omega(h(\gamma_n(t)),L^-,\Omega_1)&=&\lim\limits_{t\to1}\omega(h(\gamma_n(t))-h(k_{n_1}),L^--h(k_{n_1}),\Omega_1-h(k_{n_1}))\\
        &=&\lim\limits_{t\to1}\omega(-h(\gamma_n(t))+h(k_{n_1}),-L^-+h(k_{n_1}),\mathbb{C}\setminus(-\infty,0])\\
        &=&\lim\limits_{t\to1}\omega(\sqrt{-h(\gamma_n(t))+h(k_{n_1})},\{z\in\mathbb{C}:\arg z=-\frac{\pi}{2}\},-i\mathbb{H})\\
        &=&\lim\limits_{t\to1}\frac{\arg\sqrt{-h(\gamma_n(t))+h(k_{n_1})}+\frac{\pi}{2}}{\pi}\\
        &=&\frac{1}{2}.
    \end{eqnarray*} 
    Combining with the inequality above, we get $\lim_{t\to1}\omega(h(\gamma_n(t)),\partial\Omega^-,\Omega)\le\frac{1}{2}$. Following an almost identical procedure, but this time with the Koebe domain $\Omega_2$, we may prove that $\lim_{t\to1}\omega(h(\gamma_n(t)),\partial\Omega^+,\Omega)\le\frac{1}{2}$ as well. However, since $\partial\Omega=\partial\Omega^+\cup\partial\Omega^-$, it is obvious that
    $$\omega(h(\gamma_n(t)),\partial\Omega^-,\Omega)+\omega(h(\gamma_n(t)),\partial\Omega^+,\Omega)=1,$$
    for all $t\in[0,1)$. So, it follows that
    $$\lim\limits_{t\to1}\omega(h(\gamma_n(t)),\partial\Omega^-,\Omega)=\lim\limits_{t\to1}\omega(h(\gamma_n(t)),\partial\Omega^+,\Omega)=\frac{1}{2}.$$
    Lastly, Remark \ref{cara} once again implies the desired result.
    \end{proof}

To conclude, the combination of all the lemmas and propositions of this section provides the proof of Theorem \ref{chordaltheorem}.

\section*{Acknowledgements}

We would like to thank Alan Sola for his advice and careful consideration during the preparation of the current work. We also thank Dimitrios Betsakos for the helpful conversations.

\end{document}